\documentclass[a4paper,reqno,oneside,11pt]{amsart}
\usepackage[utf8]{inputenc}
\usepackage{a4wide}
\usepackage{amsfonts,amsmath,amsthm,eucal,url,cite}
\usepackage{enumitem}
\usepackage{tikz-cd}
\usepackage{hyperref}
\usepackage{multicol}
\usepackage{slashed}
\usepackage{appendix}
\usepackage{verbatim}

\allowdisplaybreaks[4]

\newcommand{\dd}{\mathrm{d}}

\newcommand{\RR}{\mathbb{R}}

\newcommand{\HH}{\mathbb{H}}

\newcommand{\cO}{\mathcal{O}}
\newcommand{\cC}{\mathcal{C}}

\newcommand{\cZ}{\mathcal{Z}}

\newcommand{\tr}{\mathrm{tr}}
\newcommand{\even}{\mathrm{even}}

\newcommand{\End}{\mathrm{End}}
\newcommand{\kk}{\mathfrak{k}}
\newcommand{\Id}{\mathrm{Id}}
\newcommand{\grad}{\mathrm{grad}}
\newcommand{\cCloc}{\cC^\mathrm{loc}}

\def\Scal{\mathrm{Scal}}

\def\R{\mathrm{R}}
\def\Rs{\mathcal{R}}
\def\Spin{\mathrm{Spin}}
\def\ddt{\frac{\partial}{\partial t}}
\def\Span{\mathrm{Span}}
\def\SS{\mathbb S}

\def\Cl{\mathrm{Cl}}

\def\Hess{\mathrm{Hess \,}}

\newtheorem{theorem}{Theorem}[section]
\newtheorem{prop}[theorem]{Proposition}

\newtheorem{lemma}[theorem]{Lemma}
\newtheorem{definition}[theorem]{Definition}

\theoremstyle{definition}

\theoremstyle{remark}
\newtheorem{rem}[theorem]{\bf Remark}

\begin{document}

\title{Cauchy spinors on $3$-manifolds}
\date{\today}
\author[B. Flamencourt]{Brice Flamencourt}
\address[B. Flamencourt]{Universit\'e Paris-Saclay, CNRS, Laboratoire de math\'ematiques d'Orsay, 91405, Orsay, France}
\email{brice.flamencourt@universite-paris-saclay.fr}

\author[S. Moroianu]{Sergiu Moroianu}
\address[S. Moroianu]{Institutul de Matematic\u a al Academiei Rom\^ane\\P.O.\ Box 1--764\\
RO-014700 Bucharest, Romania}
\email{moroianu@alum.mit.edu}
\begin{abstract}
Let $\mathcal{Z}$ be a  spin $4$-manifold carrying a parallel spinor and $M\hookrightarrow \cZ$ a hypersurface. The second fundamental form of the embedding induces a flat metric connection on $TM$. Such flat connections satisfy a non-elliptic, non-linear equation in terms of a symmetric $2$-tensor on $M$. When $M$ is compact and has positive scalar curvature, the linearized equation has finite dimensional kernel. Four families of solutions are known on the round $3$-sphere $\SS^3$. We study the linearized equation in the vicinity of these solutions and we construct as a byproduct an incomplete hyperk\"ahler metric on $\SS^3\times \mathbb{R}$ closely related to the Euclidean Taub-NUT metric on $\RR^4$. On $\SS^3$ there do not exist other solutions which either are constant in a left (or right) invariant frame, have three distinct constant eigenvalues, or are invariant in the direction of a left (or right)-invariant eigenvector. We deduce from this last result an extension of Liebmann's sphere rigidity theorem.
\end{abstract}

\maketitle

\section{Introduction}

\subsection*{The Cauchy problem for parallel spinors}
If $(M,g)$ is an oriented hypersurface in a spin manifold $(\mathcal Z, g_\mathcal Z)$ and $\Psi$ is a parallel spinor on $\mathcal Z$, then the restriction 
$\psi := \Psi \vert_M$ is a spinor on $M$ satisfying the initial condition
\begin{align} \label{introgks}
\nabla_X \psi = - \frac{A(X)}{2} \cdot \psi,&& (\forall) X \in T M,
\end{align}
where $A$ is the second fundamental form of $M$ (see \cite[Proposition 1.4.1]{Gin}). 

Conversely, given a spinor $\psi$ on $M$ satisfying the constraint \eqref{introgks} for some symmetric endomorphism $A$, is it always possible to embed $M$ as a hypersurface in an ambient manifold $\mathcal Z$ with second fundamental form $A$, and such that $\psi$ is the restriction to $M$ of a parallel spinor $\Psi$ on $\mathcal Z$? In other words, is \eqref{introgks} the only constraint for the existence of $\mathcal{Z}$ and $\Psi$? This is the so-called Cauchy problem for parallel spinors studied in \cite{A2M}. The answer is positive when all the objects involved are real analytic, and negative in general in the smooth setting for $\dim(M)\geq 3$. 

A spinor and symmetric $2$-tensor on $M$ satisfying Equation \eqref{introgks} will be called below a \emph{Cauchy spinor}, respectively a \emph{Cauchy endomorphism}. Since \eqref{introgks} includes as a particular case the Killing spinor equation (i.e., when $A$ is a constant multiple of the identity), Cauchy spinors were sometimes called \emph{generalized Killing spinors} e.g.\ in \cite{MS1,MS2,MS3}. Different generalizations of the notion of Killing spinors appear however in the literature, e.g.\ in the papers \cite{HM} or \cite{FK2000,FK2001}. We believe that the current name should be more appropriate, as it describes more accurately the property of being the restriction of a parallel spinor to a hypersurface. 

The classification problem for Cauchy spinors and endomorphisms on a given manifold $M$ requires us to describe all pairs $(\psi,A)$ verifying \eqref{introgks}. In dimension $3$, partial results in this direction were found in \cite{MS2}, where Cauchy spinors are characterized in terms of a triple of divergence-free vector fields on $M$. The same authors investigated in \cite{MS1,MS3} the case of the sphere $\SS^3$, classifying all Cauchy endomorphisms having at most $2$ distinct eigenvalues. This example illustrates the little understanding we have of Cauchy spinors in dimension $3$, as we are unable to classify them even on the round sphere. Note that a complete description can be given in several other dimensions \cite{MS3}.

\subsection*{Cauchy spinors on 3-manifolds and flat connections} 
Spin geometry in dimension $3$ is special because the Hodge $*$ operator allows an exceptional identification between $1$- and $2$-forms, and moreover the real spinor bundle carries a quaternionic structure. Using these algebraic structures, on simply connected $3$-manifolds (which by the Poincar\'e conjecture must be diffeomorphic to the sphere) we can restate the classification problem for Cauchy spinors without mentioning spinors at all! Indeed, Equation \eqref{introgks} implies a constraint for the symmetric endomorphism field $A$:
\begin{align} \label{introendcond}
0 = \R (X,Y) + *d^\nabla A (X, Y) + A(X) \wedge A(Y),&&(\forall)X,Y \in T M.
\end{align}
Here $\R$ is the Riemann curvature tensor and $d^\nabla$ is the exterior covariant derivative on $M$ mapping sections of $\Lambda^1 M \otimes T M$ to sections of $\Lambda^2 M \otimes T M$. The set of symmetric endomorphisms satisfying \eqref{introendcond} is denoted $\cCloc_M$.
If $M$ is simply connected, every solution of Equation \eqref{introendcond} also satisfies \eqref{introgks} for some Cauchy spinor, unique up to right multiplication by a quaternion. Our strategy below is to exploit Equation \eqref{introendcond} in order to obtain new results on the Cauchy problem for parallel spinors, and also on the classification problem for Cauchy spinors.

Equation \eqref{introendcond} amounts to the flatness of the modified metric connection $\nabla^A=\nabla+*A$ on $TM$. Even in the compact and simply connected case, the structure of the set of flat connections $\cCloc_M$ remains elusive, in part because Equation \eqref{introendcond} is non-linear and not elliptic. For this reason, we first study the linearization of \eqref{introendcond}. We show that if the scalar curvature of $M$ is positive, the space of infinitesimal deformations, defined as the space of symmetric endomorphism fields solution to the linearization of  \eqref{introendcond}, is finite-dimensional (Theorem~\ref{deformation}). This can be interpreted as a finiteness result for the dimension of the ``tangent space" of $\cCloc_M$, with the caveat that this set of flat connections is a priori not a smooth manifold. The hypothesis on the sign of the scalar curvature is necessary: we exhibit flat compact $3$-manifolds for which the dimension of the space of deformations is infinite.

\subsection*{Cauchy endomorphisms on the round three-sphere}
 We view $\SS^3$ as the Lie group of unit-length quaternions. Four examples of symmetric endomorphisms $A$ fulfilling the flatness condition \eqref{introendcond} on $\SS^3$ are known from \cite[Example 3.2]{MS2}: $\pm \mathrm{Id}$, and the endomorphism fields constant in a left (resp.\ right)-invariant orthonormal frame, with eigenvalues $1, -3, -3$ (respectively $-1, 3, 3$).
It was already shown in \cite{MS3} that there are no deformations around $A = \pm \mathrm{Id}$. We prove that the space of infinitesimal deformations around the other two examples has dimension $2$, and corresponds to the Lie derivative of $A$ in the direction of a left (or right-) invariant vector field from $\textrm{ker} (A \pm 3 \mathrm{Id})$. In particular, there are no other deformations of the above solutions. If $\cCloc_{\SS^3}$ were a smooth manifold, it would therefore necessarily have at least $4$ connected components.

The examples of endomorphisms on $\SS^3$ given above are analytic, so by \cite{A2M} they can be realized as second fundamental forms of the three-sphere embedded as a hypersurface in a generalized cylinder $\mathcal Z :=(-\epsilon,\epsilon) \times \SS^3$ carrying a parallel spinor. The cases $A = \pm \mathrm{Id}$ both induce the standard embedding of $\SS^3$ into $\RR^4$. We calculate in Section~\ref{DeformationsOnS3} an explicit expression of this metric in the other two cases, finding an extension of the family of Euclidean Taub-NUT metrics with a negative parameter. This computation solves the long-time Cauchy problem on the three-sphere for the four known examples of Cauchy spinors on $\SS^3$.

\subsection*{Classification results}
In the final section we prove three classification results for symmetric endomorphisms solving Equation~\eqref{introendcond} on $\SS^3$. The known examples of Cauchy endomorphisms have constant matrices in a left or right-invariant orthonormal basis of the tangent bundle of the Lie group $\SS^3$. It is natural to ask if those are the only symmetric endomorphisms solutions to \eqref{introendcond} with this property. We prove 
in Proposition \ref{constantA} that this is indeed the case. Moreover, since all the solutions of \eqref{introendcond} with at most two eigenvalues are known \cite[Theorem 4.10]{MS3}, we investigate the case where $A$ has three distinct constant eigenvalues. We show that there is no solution in this case, using a characterization of the Hopf fields on $\SS^3$ (Proposition~\ref{classification2}).
Finally, we show that the solutions which are constant only in the direction of a left or right-invariant eigenvector $\xi$ of $A$, i.e., such that $\mathcal L_\xi A = 0$, must also be constant in a left (respectively right) invariant orthonormal frame, so they fall in the class of previously known examples (Proposition~\ref{classification1}). This result turns out to imply (a slight extension of) the Liebmann rigidity theorem \cite{L99}.

\subsection*{Related results}
Cauchy spinors in $3+1$ Lorentzian signature were recently classified by Murcia and Shahbazi \cite{MS20}, \cite{MS21}, thanks to the fact that one of the elements in the global coframe defined by the Cauchy spinor on the hypersurface is privileged, and determines an integrable foliation. On the other hand, unlike in the Riemannian case, the Lorentzian initial value problem for parallel spinors is of hyperbolic nature and has solutions in the smooth category. Thus the two problems are in fact rather different, despite some formal similarities.

\subsubsection*{Acknowledgements} We are indebted to Andrei Moroianu for proposing the problem and for many enlightening discussions. In particular, he suggested the name ``Cauchy spinor" to replace the more traditional - but somewhat confusing - ``generalized Killing spinor". 
The second named author was partially supported from the UEFISCDI grant PN-III-P4-ID-PCE-2020-0794 ``Spectral Methods in Hyperbolic Geometry". He would like to acknowledge the hospitality of the University of Lorraine at Metz, where part of this paper was written.

\section{Preliminaries}

\subsection{Spinors in dimension $3$}

The real Clifford algebra $\Cl_2$ is canonically isomorphic to the quaternion algebra $\HH$ by the map sending $e_1$ to $i$ and $e_2$ to $j$, where $\{e_1,e_2\}$ is the standard basis of $\RR^2$. It follows that the even Clifford algebra $\Cl_3^\even$ is also isomorphic to 
$\HH$ by the unique algebra map sending $e_1e_3$ to $i$, $e_2e_3$ to $j$, and hence $e_2e_3$ to $k$.

Multiplication by the central element $P:=\frac{1-e_1e_2e_3}{2}$ is a projector in $\Cl_3$. Let $\Sigma_3$ be the image of this projector. Then $P$ maps $\Cl_3^\even$ isomorphically onto $\Sigma_3$, thus identifying $\Sigma_3$ to $\HH$. Since $P$ is central, $\Cl_3^\even$ acts on $\Sigma_3$ by left multiplication, and this representation commutes with the right action of $\HH\simeq \Sigma^3$ on itself. The restriction of this quaternionic representation to the spin group $\Spin_3$ is the spinor representation. By construction, the spinor representation is thus the restriction of a $\Cl_3$ representation under which the volume element $e_1e_2e_3$ acts as minus the identity. 

The spinor representation is orthogonal with respect to the natural scalar product on $\Sigma_3$. The right multiplication by quaternions is also compatible with the scalar product, in the sense that $\langle \psi a, \phi a \rangle = \vert a \vert^2 \langle \psi , \phi \rangle$ for all $a \in \HH$ and $\psi, \phi \in \Sigma_3$.

Recall now that every oriented $3$-manifold is parallelizable, hence spin. Let $(M,g)$ be an oriented Riemannian 3-manifold with a fixed spin structure. The real spinor bundle $\Sigma M$ over $M$ is the vector bundle associated to the spin bundle and the $\Spin_3$ spinor representation $\Sigma_3$. It is endowed with a natural inner product $\langle \cdot, \cdot \rangle$ and a $\Cl(TM)$-action such that the Clifford product by vectors on $\Sigma M$ is skew-symmetric. The right action of $\HH$ on $\Sigma_3$ induces a right $\HH$ action on $\Sigma M$ commuting with Clifford multiplication by tangent vectors, and satisfying $\langle \psi a, \phi a \rangle = \vert a \vert^2 \langle \psi , \phi \rangle$ for all $a \in \HH$ and $\psi, \phi \in \Sigma M$.

By construction, Clifford multiplication with the volume form acts as $- \mathrm{Id}$ on $\Sigma M$. This choice implies that for a $2$-form $\omega \in \Lambda^2(M)$ and a spinor $\psi \in \Sigma M$ we have
\begin{equation} \label{2form}
\omega \cdot \psi = * \omega \cdot \psi,
\end{equation}
where $*$ denotes the Hodge star-operator.

\subsection{Cauchy spinors and endomorphisms}
Let $E\to M$ be a vector bundle endowed with a connection $\nabla$.
The exterior differential twisted by $\nabla$ is defined on $E$-valued $p$-forms using Einstein's summation convention:
\begin{equation*}
d^\nabla (\omega \otimes V) := d \omega \otimes V + (-1)^p \omega \wedge X^j \otimes \nabla_{X_j} V,
\end{equation*}
where $(X_j)_{1 \leq j \leq n}$ is any local orthonormal frame on $M$. 
For an endomorphism field $A\in\End(TM)$, the above formula becomes
\begin{align} \label{extcovder}
d^\nabla A (X,Y) := (\nabla_X A) Y - (\nabla_Y A) X,&& (\forall) X,Y\in TM.
\end{align}
The divergence operator $\delta^\nabla$ is the formal adjoint of $d^\nabla$ with respect to the $L^2$-inner product on vector-valued forms.

Let now $\nabla$ be the Levi-Civita covariant derivative on an oriented $3$-manifold $M$, and $\R=d^\nabla\circ\nabla$ its Riemann curvature tensor.
The Levi-Civita covariant derivative on the spinor bundle $\Sigma M$ is still denoted by $\nabla$, and its curvature tensor is written $\Rs$. 

\begin{definition}
Let $(M,g)$ be a spin $3$-manifold.
A non-zero section $\psi \in \Gamma (\Sigma M)$ is a \emph{Cauchy spinor} if there exists a symmetric endomorphism field $A\in\Gamma(T^*M\times TM)$ such that the pair $(\psi,A)$ satisfies Equation \eqref{introgks}. In this situation, $A$ is called a \emph{Cauchy endomorphism}.

We denote by $\cC_M$ the set of all Cauchy endomorphisms on $M$, and by $\cCloc_M$ the set of all symmetric endomorphisms on $M$ satisfying Equation \eqref{introendcond}.
\end{definition}
In other words, an endomorphism field $A$ on $M$ is a Cauchy endomorphism if it is symmetric and there exists some non-zero spinor 
$\psi\in\Gamma(\Sigma M)$ satisfying $\nabla_X\psi=-\frac{1}{2}A(X)\cdot\psi$ for all vectors $X\in TM$.
We stress that the symmetry assumption on $A$ is crucial here, since in dimension $3$ every spinor of constant length determines uniquely some general endomorphism field $A$ so that Equation \eqref{introgks} holds. 

The sets $\cC_M$ and $\cCloc_M$ will form our main object of study in this paper.

\subsection{Parallel spinors in dimension $4$}

Let us review below the main results of \cite{A2M} about parallel spinors in dimension $4$, Ricci-flat metrics and Cauchy spinors.

Let $(M,g)$ be a hypersurface in a $4$-dimensional manifold $(\cZ,g^{\mathcal Z})$. 
If $\mathcal{Z}$ is Ricci-flat, the second fundamental form $W$ of the embedding $M\hookrightarrow \mathcal{Z}$ satisfies the contracted Codazzi and Gauss equations:
\begin{align}\label{cgc}
\Scal^M = \tr(A)^2-\tr(A^2),&&\delta^{\nabla}A+d\tr(A)=0.
\end{align}
Conversely, if $(M,g^M)$ is real analytic and the constraints \eqref{cgc} hold on $M$ for some analytic symmetric endomorphism $A$, then there exists a Ricci-flat real-analytic ambient manifold $(\mathcal Z,g^{\mathcal Z})$ in which $M$ imbeds isometrically with second fundamental form $A$.

Upon replacing $\cZ$ by a collar neighborhood of $M$, we can assume that $\cZ$ is also parallelizable. Fix a spin structure on $\cZ$. The restriction of each of the spinor bundles $\Sigma^\pm \cZ$ to $M$ is isomorphic to $\Sigma M$.
If $\mathcal{Z}$ admits a nonzero parallel spinor $\Psi$, then $g^{\mathcal Z}$ must be Ricci flat. Up to reversing the orientation on $\cZ$ we can assume that the chiral component $\Psi^+$ is nonzero, implying that the metric is self-dual (if $\Psi^+$ and $\Psi^-$ are both nonzero then $\cZ$ is flat). Moreover, the restriction of $\Psi^+$ to $M$ is a Cauchy spinor, and the second fundamental form of $M\hookrightarrow \cZ$ is a Cauchy endomorphism. 

Conversely, if there exists a Cauchy spinor $\psi$ and a Cauchy endomorphism $A$ on $M$ satisfying \eqref{introgks}, then the identities \eqref{cgc} are satisfied (an alternate derivation of these identities can be found in \cite[Lemma 3.1, Equation (12)]{MS1}). If moreover $M$, $g^M$, $A$ and $\psi$ are real-analytic, then there exists a Ricci-flat, self-dual $4$-manifold $(\cZ,g^{\mathcal Z})$ into which $(M,g^M)$ embeds isometrically with second fundamental form $A$, and locally there exists a parallel spinor of positive chirality on $\cZ$ extending $\psi$.

\subsection{The modified metric connection $\nabla^A$}
Throughout the paper we identify $1$-forms and vectors on $M$ using the metric $g$. 
Given any $A \in \Gamma( \Lambda^1 M \otimes T M)$, we define $\overline A \in \Gamma(\Lambda^1 M \otimes \Lambda^2 M)$ by 
$\overline A (X) := *(A (X))$ for all $X \in T M$. We can view the $2$-form $\overline A(X)$ as a skew-symmetric endomorphism of $T M$ in the usual way: for $Y \in T M$, $\overline A (X) (Y)$ is the unique vector $Z$ satisfying $g(Z,W) = \overline A (X) (Y,W)$ for all $W \in T M$. 

We introduce the connection on $TM$
\begin{align} \label{modifiedconnection}
\nabla^A_X Y := \nabla_X Y + \overline A (X) (Y),&&X,Y \in T M,
\end{align}
and we denote by $\R^A$ its curvature tensor. Since by construction $\overline A (X)$ is skew symmetric, $\nabla^A$ is compatible with the Riemannian metric and hence induces a connection, still denoted $\nabla^A$, on the spinor bundle $\Sigma M$ 
by pull-back from the orthonormal frame bundle:
\begin{align*}
\nabla^A_X \psi = \nabla_X \psi + \frac{A(X)}{2} \cdot \psi,&&(\forall) X \in TM,\psi \in \Sigma M.
\end{align*}
We denote by $\Rs^A$ the curvature of $\nabla^A$ on spinors.

The problem of finding solutions to \eqref{introgks} can be reduced to solving an equation involving the endomorphism field $A$ alone, at least when $M$ is simply-connected.

\begin{prop} \label{endCond}
Let $(M,g)$ be a Riemannian $3$-manifold and $A \in \Gamma( \Lambda^1 M \otimes T M)$ a symmetric endomorphism field. The following conditions are equivalent:
\begin{enumerate}
\item Locally on $M$ there exist nonzero Cauchy spinors with respect to $A$;
\item $\Rs^A=0$;
\item $\R^A=0$;
\item The symmetric endomorphism $A$ satisfies the equation \eqref{introendcond}, i.e., $A\in\cCloc_M$.
\end{enumerate}
When $M$ is simply connected, the first condition is equivalent to the global existence of Cauchy spinors, hence $\cCloc_M=\cC_M$.
\end{prop}
\begin{proof}
If there exists a locally defined, nonzero spinor $\psi$ satisfying \eqref{introgks}, then the four mutually orthogonal spinors $(\psi, \psi i, \psi j, \psi k)$, defined with the help of the quaternionic structure on $\Sigma M$, are parallel for the covariant derivative $\nabla^A$. This implies that the curvature endomorphism $\Rs^A$ of $(\Sigma M, \nabla^A)$ vanishes.
Conversely, if $\Rs^A = 0$, locally there exist non-zero spinors parallel with respect to the connection $\nabla^A$, i.e., Cauchy spinors. Moreover, if $M$ is simply connected then such spinors exist globally on $M$.

Since the connection $\nabla^A$ on $\Sigma M$ is induced from the connection with the same name on $TM$ and the spinor representation $\Sigma M$, it is well known that 
$\Rs^A(X,Y)=\frac{1}{2}\R^A(X,Y)\cdot$, where $\cdot$ denotes Clifford multiplication (see e.g.\ \cite[Theorem 2.7]{BH3M}). In dimension $3$ the map associating to a $2$-form its action by Clifford multiplication on spinors is injective, thus $\Rs^A(X,Y)=0$ if and only if $\R^A(X,Y)=0$.

For all $X,Y \in T M$, we compute from the definition of $\nabla^A$ its curvature:
\begin{equation*}
\R^A (X,Y) = \R (X,Y) + d^\nabla \overline{A}(X,Y) + [\overline A(X), \overline A(Y)]
\end{equation*}
Since the Hodge $*$ operator is parallel, it commutes with $d^\nabla$, so $d^\nabla \overline{A}=*d^\nabla A$. Also, we check directly that
\[
\overline A(X) \circ \overline A(Y) - \overline A(Y) \circ \overline A(X) = A(X) \wedge A(Y).
\]
Hence $\R^A=\R+*d^\nabla A+A\wedge A$ as claimed.
\end{proof}

Proposition~\ref{endCond} shows that $\cC_M \subset \cCloc_M$, with equality when $M$ is simply connected.

\section{Deformation of endomorphism fields}

There is not much one can say about the structure of the sets $\cC_M$ or $\cCloc_M$, even in the simply connected case when they coincide. Let us introduce the following definition:
\begin{definition} \label{defordef}
The space of infinitesimal deformations of $\cCloc_M$ at $A \in \cCloc_M$ is
\begin{align*}
\lbrace \dot C (0) \vert \: (\exists) \varepsilon > 0,
(\exists) C \in C^\infty ((- \varepsilon, \varepsilon), \Gamma (\mathrm{Sym}^2(T M)), C(0) = A, \frac{\dd}{\dd t} \R^{C (t)} \vert_{t = 0} = 0 \rbrace.
\end{align*}
\end{definition}

In the case where $\cCloc_M$ is a differentiable manifold, its tangent space $T_A\cCloc_M$ is the space of tangent vectors at $t=0$ to smooth curves $C : (-\varepsilon, \varepsilon) \to \cCloc_M$ with $C(0) = A$. In general, since we do not know any differentiable structure on $\cCloc_M$, the above cone might even not be a vector space. The space of infinitesimal deformations always contains this formal tangent cone.

\begin{theorem} \label{deformation}
Let $(M,g)$ be a compact oriented Riemannian $3$-manifold with strictly positive scalar curvature and let $A \in \cCloc_M$. Then the space of infinitesimal deformations of $A$ is finite-dimensional.
\end{theorem}
\begin{proof}
Let $C \in C^\infty ((- \varepsilon, \varepsilon), \Gamma (\mathrm{Sym}^2(T M))$ be a smooth curve as in Definition~\ref{defordef}, with $C(0) = A$. We define $\dot A := \dot C (0)$. By differentiating with respect to $t$ in \eqref{introendcond} the condition $\frac{\dd}{\dd t} \R^{C (t)} \vert_{t = 0} = 0$ rewrites
\begin{equation} \label{eqder}
0 = *(\dd^{\nabla} \dot A) (X,Y) + \dot A(X) \wedge A(Y) + A(X) \wedge \dot A (Y).
\end{equation}
We recall some elementary identities about the Hodge star-operator in dimension $3$.
Let $X,Y \in T M$ and $\alpha \in \Lambda^2 M$. Then we easily check that
\begin{align}\label{HodgeFormulas}
X \lrcorner * Y &= - * (X \wedge Y) ,&
X \wedge * \alpha &= - * (X \lrcorner \alpha).
\end{align}

Using these identities, one has for all $X,Y \in T M$
\begin{align*}
\dot A (X) \wedge A(Y) + A(X) \wedge \dot A(Y) &= - *\overline A(Y) (\dot A (X)) + *\overline A(X) (\dot A (Y)).
\end{align*}
Consequently, Equation \eqref{eqder} rewrites
\begin{align*}
0 &= (\dd^{\nabla} \dot A) (X,Y) - \overline A(Y) (\dot A (X)) + \overline A(X) (\dot A (Y)) 
= (d^{\nabla^A} \dot A) (X,Y).
\end{align*}
We are hence led to the equation for the infinitesimal deformation of a flat connection:
\begin{equation}\label{fc}
d^{\nabla^A} \dot A = 0.
\end{equation}

According to Proposition~\ref{endCond}, the connection $\nabla^A$ is flat, meaning that $\R^A=d^{\nabla^A}\circ\nabla^A=0$ on $\Gamma(TM)$. But $d^{\nabla^A} \circ d^{\nabla^A}$ is given by the action of $\R^A$ also on $\Gamma(\Lambda^kM\otimes TM)$ for all $k
\geq 0$.
This implies that $d^{\nabla^A} \circ d^{\nabla^A} = 0$ on $\Gamma(\Lambda^*M\otimes M)$. The differential operators 
\[d^{\nabla^A}:\Gamma(\Lambda^*M\otimes TM)\to \Gamma(\Lambda^{*+1}M\otimes TM)\] 
form therefore an elliptic complex.
By Hodge theory, Equation \eqref{fc} implies that there exists a vector field $X_{\dot A} \in \Gamma(T M)$ and a $d^{\nabla^A}$-harmonic vector-valued $1$-form $B_{\dot A} \in \Gamma(\Lambda^1 M \otimes TM)$ (i.e., $\delta^{\nabla^A}B_{\dot A} = 0$ and $d^{\nabla^A} B_{\dot A} = 0$) such that $\dot A = \nabla^A X_{\dot A} + B_{\dot A}$. Notice that this equation still involves the symmetric endomorphism $A$ which defines the flat connection 
$\nabla^A$. 

The symmetry condition on $\dot{A}$ can be rewritten as an equation on $X_{\dot A}$, $B_{\dot A}$ and $A$:

\begin{lemma} \label{symcond}
Let $X \in \Gamma(T M)$ and $B \in \Gamma(\Lambda^1 M \otimes T M)$. The endomorphism $d^{\nabla^A} X + B$ is symmetric if and only if $X$, viewed as a 1-form, satisfies $dX - *(X\mathrm{tr}(A)-AX) + X_k \wedge B (X_k) = 0$, where $(X_1, X_2, X_3)$ is any orthonormal basis.
\end{lemma}
\begin{proof}
The endomorphism $d^{\nabla^A} X + B$ is symmetric if and only if its skew-symmetric part is zero. Using the identities \eqref{HodgeFormulas} we compute
\begin{align*}
0 &= X_k \wedge \nabla^A_{X_k} X + X_k \wedge B(X_k) \\
&= X_k \wedge (\nabla_{X_k} X + \overline A (X_k) (X)) + X_k \wedge B(X_k) \\
&= X_k \wedge (\nabla_{X_k} X + X \lrcorner \overline A (X_k)) + X_k \wedge B(X_k) \\
&= X_k \wedge (\nabla_{X_k} X - * (X \wedge A(X_k))) + X_k \wedge B(X_k) \\
&= X_k \wedge \nabla_{X_k} X - X_k \wedge * (X \wedge A(X_k)) + X_k \wedge B(X_k) \\
&= X_k \wedge \nabla_{X_k} X + * (X_k \lrcorner X \wedge A(X_k)) + X_k \wedge B(X_k) \\
&= d X + *(X (X_k) A(X_k) - A(X_k)(X_k) X) + X_k \wedge B(X_k) \\
&= d X + *(A(X) - X \textrm{tr} A) + X_k \wedge B(X_k).\qedhere
\end{align*}
\end{proof}

The space of $d^{\nabla^A}$-harmonic vector-valued $1$-forms $B$ is finite dimensional by ellipticity. Lemma~\ref{symcond} implies the identity $d X_{\dot A} - *(X_{\dot A} \textrm{tr}(A)-AX_{\dot A}) = - X_k \wedge B_{\dot A} (X_k)$, and for a given $B$, the solutions $X \in TM$ of
\[
dX - *(X\mathrm{tr}(A)-AX) + X_k \wedge B (X_k) = 0
\]
form an affine space of direction $\mathrm{ker} (X \mapsto dX - *(X\mathrm{tr}(A)-AX))$. Thus, to show that the space of deformations is finite dimensional, it is sufficient to prove that the solution space of 
\begin{equation} \label{kerX}
dX - *(X\mathrm{tr}(A)-AX) = 0
\end{equation}
is finite dimensional.

By applying the exterior derivative to equation~\eqref{kerX} we have $0 = d * (X\textrm{tr}(A)-AX) = * \delta (X\textrm{tr}(A)-AX)$, and then
\begin{equation*}
\delta (X\textrm{tr}(A)-AX)=0.
\end{equation*}
We define the differential operator
\begin{align*}
\Xi : \Lambda^1 M & \longrightarrow  \Lambda^2 M \oplus \Lambda^0 M,&
\Xi(X)=(d X - * (X\textrm{tr}(A)-AX), \delta (X \textrm{tr} A - A X)).
\end{align*}
To compute the principal symbol of $\Xi$, let $f$ be a smooth function on $M$. Again by using \eqref{HodgeFormulas} we obtain
\begin{align*}
\sigma_\Xi (d f)(X) &= \Xi (f X) - f \Xi X \\
&= (d (f X) - f d X, \delta [f (X \textrm{tr} A - A X)] - f \delta [X \textrm{tr} A - A X]) \\
&= (d f \wedge X, -\langle d f, (X \textrm{tr} A - A X)\rangle).
\end{align*}
Hence, the principal symbol of $\Xi$ is given by
\begin{equation*}
\sigma_\Xi (\alpha) (X) = (\alpha \wedge X, \langle\alpha, (X \textrm{tr} A - A X)\rangle).
\end{equation*}

We want to show that this principal symbol is injective in order to use the theory of elliptic operators. To do so, we first remark that by equation~\eqref{cgc}, the eigenvalues $(\lambda_1, \lambda_2, \lambda_3)$ of $A$ satisfy $\lambda_1 \lambda_2 + \lambda_1 \lambda_3 + \lambda_2 \lambda_3 = \frac{1}{2}\Scal^M$. We now use the hypothesis $\Scal^M>0$.

\begin{lemma} \label{algtech}
Let $B \in M_3(\RR)$ be a symmetric matrix with eigenvalues $(\lambda_1, \lambda_2, \lambda_3)$ such that $\lambda_1 \lambda_2 + \lambda_1 \lambda_3 + \lambda_2 \lambda_3 > 0$. Then $B - \mathrm{tr}(B) \mathrm{Id}$ is definite.
\end{lemma}
\begin{proof}
Since $B$ is symmetric and real, there is $P \in \mathrm{O}_3(\RR)$ such that
\[
P^T B P = \mathrm{diag} (\lambda_1, \lambda_2, \lambda_3).
\]
Thus, $P^T (B - \mathrm{tr}(B) \mathrm{Id}) P = - \mathrm{diag} (\lambda_2 + \lambda_3, \lambda_1 + \lambda_3, \lambda_1 + \lambda_2)$. Moreover, one has
\begin{align*}
(\lambda_2 + \lambda_3) (\lambda_1 + \lambda_3) &= \lambda_2 \lambda_1 + \lambda_2 \lambda_3 + \lambda_3 \lambda_1 + \lambda_3^2 > 0 
\end{align*}
and similarly $(\lambda_1 + \lambda_3) (\lambda_1 + \lambda_2)  > 0$, 
$(\lambda_2 + \lambda_3) (\lambda_1 + \lambda_2) > 0$.
We conclude that $\lambda_2 + \lambda_3, \lambda_1 + \lambda_3, \lambda_1 + \lambda_2$ have the same sign and are different from 0, so $B - \mathrm{tr}(B) \mathrm{Id}$ is definite.
\end{proof}

As a consequence of Lemma~\ref{algtech}, $A - \textrm{tr} (A) \textrm{Id}$ is definite under the assumption that $\Scal^M>0$. Now, let $\alpha \in \Lambda^1 M$ such that there is a non-zero vector $X$ with $\sigma_\Xi (\alpha) (X) = 0$. In particular, we have $g(\alpha, [A - \mathrm{tr}(A) \mathrm{Id}] X) = 0$ and $\alpha \wedge X = 0$. We deduce that $X = f \alpha$ and $g(\alpha, [A - \mathrm{tr}(A) \mathrm{Id}] \alpha) = 0$, so $\alpha = 0$ because the endomorphism $A - \textrm{tr} (A) \textrm{Id}$ is invertible. Consequently, the principal symbol of $\Xi$ is injective.

The operator $\Xi^* \Xi$ has the same kernel as $\Xi$ and is elliptic, so its kernel is finite dimensional (see e.g.\ \cite[Theorem 5.2]{LM}). Thus the space of infinitesimal deformations of $A$ is finite dimensional, ending the proof.
\end{proof}

The assumption $\Scal^M > 0$ is necessary, as shown by the following example:

\begin{rem}
We look at a flat Riemannian product $M=\SS^1 \times E$, where $E=\RR^2/\Gamma$ is an elliptic curve. Let $p:M \rightarrow \SS^1$ be the projection on the first factor, and $P:T M\rightarrow TM$ the orthogonal projection on the first factor in the tangent bundle. The endomorphism $P$ is parallel and symmetric, and it clearly satisfies Equation \eqref{introendcond}. Moreover, for any smooth function $f:\SS^1 \to \RR$, the symmetric endomorphism field $A_f:=(p^*f)P$ also satisfies Equation \eqref{introendcond} since all the terms in this equation vanish when at least one of the vectors $X,Y$ are tangent to the second factor $E$. Thus the space of infinitesimal deformations of $A_f$ contains the infinite-dimensional space $C^\infty(\SS^1)$.

Not every such $A_f\in\cCloc_M$ is necessarily associated to a Cauchy spinor, because the torus is not simply connected. Take a non-zero parallel spinor $\Psi$ on the flat torus $F \times E$, where $F$ is also an elliptic curve, and consider any closed simple curve $\gamma$ in $F$ of length $2\pi$, hence isometric to $\SS^1$. Then the manifold $\gamma \times E$ is isometric to $M$, so the restriction of $\Psi$ to $M$ is a Cauchy spinor on $M$ with respect to $A_\kk$, where $\kk$ is the geodesic curvature function of $\gamma$. Since the set of curvature functions of curves of length $2\pi$ in $F$ parametrized by arc-length is not finite-dimensional, it is evident that the deformation space of flat connections near such a $A_\kk$ cannot be finite dimensional either.
\end{rem}

\section{Deformations of Cauchy spinors on the three-sphere} \label{DeformationsOnS3}

Let us illustrate the above result in the case of the round $3$-sphere, noting that even in this simplest possible case Equation \eqref{introgks} is not yet fully understood.

We identify $\SS^3$ with the unit sphere in the quaternions $\HH \simeq \RR^4$. In this way, $\SS^3$ becomes a Lie group with Lie algebra the space of imaginary quaternions $\textrm{Im} \HH$. Let $(e_1, e_2, e_3)$ be the three left-invariant vector fields corresponding to the quaternions $i,j,k$. They form an orthonormal frame at any point of $\SS^3$. Recall that 
the Levi-Civita covariant derivative of
left-invariant vector fields on compact Lie groups is given by $\nabla_X Y = \frac{1}{2}[X,Y]$. Recall also that for an even permutation $(a,b,c)$ of the indices $(1,2,3)$, the Lie brackets are given by 
\begin{align}\label{crej}
[e_a,e_b]=2e_c.
\end{align}
We compute from here the covariant derivatives of these orthonormal vector fields:
\begin{align}
&\nabla_{e_1} e_1 = 0&&\nabla_{e_2} e_2 = 0&&\nabla_{e_3} e_3 = 0\nonumber \\
&\nabla_{e_1} e_2 = e_3 && \nabla_{e_2} e_3 = e_1 && \nabla_{e_3} e_1 = e_2   \label{covE}\\
&\nabla_{e_2} e_1 = - e_3 && \nabla_{e_3} e_2 = - e_1&&\nabla_{e_1} e_3 = - e_2.\nonumber
\end{align}
On the round sphere, the curvature tensor satisfies $\R(X,Y)=-X\wedge Y$ so
Equation \eqref{introendcond} rewrites
\begin{equation} \label{endsphere}
* d^\nabla A (X,Y) = X \wedge Y - A (X) \wedge A (Y).
\end{equation}
\begin{rem} \label{exEnd}
From \cite[Example 3.2]{MS2}, we know four families of Cauchy endomorphisms in $\cC_{\SS^3}$ :
\begin{itemize}
\item plus or minus the identity
\item the symmetric endomorphism fields constant in a left-invariant orthonormal frame, with eigenvalues 1, -3, -3
\item the symmetric endomorphism fields constant in a right-invariant orthonormal frame, with eigenvalues -1, 3, 3.
\end{itemize}
\end{rem}

It was already shown in \cite[Theorem 5.1]{MS3} that $\cCloc_{\SS^3}$ does not admit infinitesimal deformations at the endomorphisms $\pm \textrm{Id}$. Let us thus study the infinitesimal deformations of $\cCloc_{\SS^3}$ at the symmetric endomorphism field
\begin{equation}\label{defA0}
A_0 := e_1 \otimes e_1 - 3 (e_2 \otimes e_2 + e_3 \otimes e_3).
\end{equation}

\begin{lemma} \label{Cohom}
Let $(M,g)$ be a simply connected $3$-manifold and $A\in\Gamma(\End(TM))$ such that the connection $\nabla^A$ defined by \eqref{modifiedconnection} is flat. Then the cohomology space $H^1(M,d^{\nabla^A})$ vanishes, hence there are no nonzero $d^{\nabla^A}$-harmonic sections in $\Gamma(\Lambda^1M\otimes TM)$. 
\end{lemma}
\begin{proof}
Fix a  global frame $s_1,s_2,s_3\in\Gamma(TM)$ consisting of $\nabla^A$-parallel vector fields, possible since $M$ is simply-connected.
In this basis, the elliptic complex $(\Gamma(\Lambda^*M\otimes TM),d^{\nabla^A})$ is isomorphic to the tensor product of the standard
de Rham complex with $\RR^3$. It follows that $H^1(M,d^{\nabla^A})\simeq H^1(M)\otimes \RR^3$, and this space vanishes since the first Betti number of a simply-connected manifold is zero.
\end{proof}
By the analysis from the proof of Theorem \ref{deformation} and Lemma~\ref{Cohom}, any infinitesimal deformation $\dot A$ can be written as the covariant derivative $\nabla^{A_0}$ of a vector field $X_{\dot A}$. By Lemma~\ref{symcond}, the symmetry of $\dot{A}$ leads to the equation
\begin{equation} \label{eqA0}
d X + *(A_0 X + 5 X)=0.
\end{equation}

If we write $X =: x^k e_k$, we have $d X = d (x^k e_k) = d x^k \wedge e_k + x^k d e_k$ and the exterior derivatives are given by $d e_k = - 2 *e_k$ for all $k \in \{1,2,3\}$.
Finally, equation~\eqref{eqA0} rewrites
\[
d x^k \wedge e_k + 4 e_2 \wedge e_3 = 0.
\]

This means that we have the differential system of equations in the unknown functions $x^1,x^2,x^3\in C^\infty(\SS^3)$:
\begin{align*}
e_2 (x^1) = e_1 (x^2), && e_3 (x^1) = e_1 (x^3),&&
e_3 (x^2) = e_2 (x^3) + 4 x^1.
\end{align*}
By taking further partial derivatives one has
\begin{align*} 
e_3 e_2 (x^1) &= e_3 e_1 (x^2) = e_1 e_3 (x^2) - 2 e_2 (x^2) \\
e_2 e_3 (x^1) &= e_2 e_1 (x^3) = e_1 e_2 (x^3) - 2 e_3 (x^2)\\
4 e_1 (x^1) &= e_1 e_3 (x^2) - e_1 e_2 (x^3).
\end{align*}
We subtract the first equation from the sum of the last two, and we obtain
\[
3 e_1 (x^1) + e_2 (x^2) + e_3 (x^3) = 0.
\]
Hence,
\[
3 e_1 e_1 (x^1) = - e_1 e_2 (x^2) - e_1 e_3 (x^3), \quad e_2 e_2 (x^1) = e_2 e_1 (x^2), \quad e_3 e_3 (x^1) = e_3 e_1 (x^3),
\]
and summing these three equations one has
\[
3 e_1 e_1 (x^1) + e_2 e_2 (x^1) + e_3 e_3 (x^1) = 2 e_2 (x^3) - 2 e_3 (x^2) = - 8 x^1.
\]
Consequently, we have to solve
\[
\Delta_{B} x^1 = 8 x^1
\]
where $\Delta_{B} := - (3 e_1 e_1 + e_2 e_2 + e_3 e_3)$ is the Laplacian on the Berger sphere with metric $\frac{1}{3} e_1^2 + e_2^2 + e_3^2$. From \cite[Section 6.2]{BBB} we know that the multiplicity of the eigenvalue 8 of $\Delta_B$ is 3, and the associated eigenspace $V_8$ is spanned by the functions $\pi^* y_k$ for $k \in \{ 1,2,3 \}$ where $\pi : \SS^3 \rightarrow \SS^2 (\frac{1}{2})$ is the Hopf fibration for which $e_1$ is tangent to the fibers and $y_k$ is the $k^\text{th}$ coordinate in $\RR^3 \supset \SS^2$. In particular, the action of $e_1$ is trivial on $V_8$, so $e_1(x^1)=0$.

\begin{lemma} \label{derivatives}
For any $k \in \{ 1,2,3 \}$ one has $e_2 e_3 (\pi^* y_k) = e_3 e_2 (\pi^* y_k) = 0$ and $e_2 e_2 (\pi^* y_k) = e_3 e_3 (\pi^* y_k) = -4 \pi^* y_k$.
\end{lemma}
\begin{proof}
The space $\Span (\pi^* y_1, \pi^* y_2, \pi^* y_3)$ is generated by the three harmonic quadratic polynomials
$\{
a_1^2 + a_2^2 - a_3^2 - a_4^2, a_1 a_4 + a_2 a_3, a_1 a_3 - a_2 a_4
\}$
restricted to $\SS^3$, where $a_k$ stands for the $k^\text{th}$ coordinate in $\RR^4$. The lemma follows by a direct computation.
\end{proof}

Now we have
\begin{align*}
e_1 e_1 (x^2) &= e_1 e_2 (x^1) = e_2 e_1 (x^1) + 2 e_3 (x^1) = 2 e_3 (x^1) \\
e_2 e_2 (x^2) &= - e_2 e_3 (x^3) \\
e_3 e_3 (x^2) &= e_ 3 e_2 (x^3) + 4 e_3 (x^1) 
\end{align*}
and by adding these equations we obtain for the Laplacian $\Delta=- (e_1 e_1 + e_2 e_2 + e_3 e_3)$ of the round metric:
\begin{equation*}
\Delta x^2 = - 2 e_3 (x^1) + 2 e_1 (x^3) - 4 e_3 (x^1) = - 4 e_3 (x^1).
\end{equation*}

In the same way we have
\begin{align*}
e_1 e_1 (x^3) &= e_1 e_3 (x^1) = - 2 e_2 (x^1) \\
e_2 e_2 (x^3) &= e_2 e_3 (x^2) - 4 e_2 (x^1) \\
e_3 e_3 (x^3) &= - e_3 e_2 (x^2),
\end{align*}
thus,
\begin{equation*}
\Delta x^3 = 2 e_2 (x^1) - 2 e_1 (x^2) + 4 e_2 (x^1) = 4 e_2 (x^1).
\end{equation*}
We are left with the system
\begin{align*}
\Delta x^2 = - 4 e_3 (x^1), && \Delta x^3 = 4 e_2 (x^1).
\end{align*}

Since $e_1, e_2, e_3$ are Killing vector fields, $\Delta$ commutes with $e_1,e_2$ and $e_3$. 
This implies that 
\begin{align*}
\Delta(x^2+\tfrac{1}{2}e_3(x^1))=0,&&\Delta(x^3-\tfrac{1}{2}e_2(x^1))=0.
\end{align*}
Since harmonic functions on a compact manifold must be constant, we deduce that $x_2 = - \frac{1}{2} e_3 (x^1) + c_2$ and $x_3 = \frac{1}{2} e_2 (x^1) + c_3$ for some constants $c_2 , c_3 \in \RR$. Finally, the space of solutions 
\begin{equation} \label{defor}
\mathcal S := \{x^1 e_1+ (- \tfrac{1}{2}e_3 (x^1) + c_2)e_2+ (\tfrac{1}{2}e_2 (x^1) + c_3)e_3; x^1\in V_8, \: c_2,c_3\in \RR\}
\end{equation}
is $5$-dimensional, since the eigenspace $V_8$ of $\Delta_B$ for the eigenvalue $8$ has dimension $3$.

From \eqref{crej}, the Lie derivatives of $A_0$ in the direction of $e_2,e_3$ are given by
\begin{align*}
\mathcal L_{e_2} A_0 &= \mathcal L_{e_2} (e_1 \otimes e_1 - 3 e_2 \otimes e_2 - 3 e_3 \otimes e_3) \\
&= -2 (e_3 \otimes e_1 + e_1 \otimes e_3) - 6 (e_1 \otimes e_3 + e_3 \otimes e_1) \\
&= - 8 (e_3 \otimes e_1 + e_1 \otimes e_3); \\
\mathcal L_{e_3} A_0 &= \mathcal L_{e_3} (e_1 \otimes e_1 - 3 e_2 \otimes e_2 - 3 e_3 \otimes e_3) \\
&= 2 (e_2 \otimes e_1 + e_1 \otimes e_2) + 6 (e_1 \otimes e_2 + e_2 \otimes e_1) \\
&= 8 (e_2 \otimes e_1 + e_1 \otimes e_2).
\end{align*}
Let $X := x^k e_k \in \mathcal S$. One has
\begin{align*}
{\nabla^{A_0}} X ={}& d x^k \otimes e_k + x^k \nabla^{A_0} e_k \\
={}& d x^k \otimes e_k + 2 x^1 (e_2 \otimes e_3 - e_3 \otimes e_2) \\
{}&+ 2 x^2 (e_1 \otimes e_3 + e_3 \otimes e_1) - 2 x^3 (e_1 \otimes e_2 + e_2 \otimes e_1),
\end{align*}
and we can compute the coefficients of ${\nabla^{A_0}} X$ using Lemma~\ref{derivatives}:
\begin{align*}
\langle {\nabla^{A_0}} X, e_1 \otimes e_1 \rangle &= 2 e_1 (x^1) = 0 \\
\langle {\nabla^{A_0}} X, e_2 \otimes e_2 \rangle &= - e_2 e_3 (x^1) = 0 \\
\langle {\nabla^{A_0}} X, e_3 \otimes e_3 \rangle &= e_3 e_2 (x^1) = 0 \\
\langle {\nabla^{A_0}} X, e_1 \otimes e_2 \rangle &= - e_1 e_3 (x^1) - 2 e_2 (x^1) - 2 c_2 = 2 e_2 (x^1) - 2 e_2 (x^1) - 2 c_2 = - 2 c_2 \\
\langle {\nabla^{A_0}} X, e_1 \otimes e_3 \rangle &= -2 e_3 (x^1) + e_1 e_2 (x^1) + 2 c_1 = -2 e_3 (x^1) + 2 e_3 (x^1) + 2 c_1 = 2 c_1 \\
\langle {\nabla^{A_0}} X, e_2 \otimes e_3 \rangle &= e_2 e_2 x^1 + 4 x^1 = 0.
\end{align*}
From the symmetry of $\nabla^{A_0} X$ we conclude that
\begin{align*}
\nabla^{A_0} X &= 2 c_1 (e_1 \otimes e_3 + e_3 \otimes e_1) - 2 c_2 (e_1 \otimes e_2 + e_2 \otimes e_1) \\
&= - \frac{1}{4} (c_1 \mathcal L_{e_3} A_0 + c_2 \mathcal L_{e_2} A_0).
\end{align*}
Consequently, we proved the following proposition:

\begin{prop} \label{deforA}
The space of infinitesimal deformations of $\cCloc_{\SS^3}$ at $A_0$ is of dimension 2, and consists exactly of the Lie derivatives of $A_0$ in the directions spanned by $(e_2,e_3)$.
\end{prop}

The same analysis holds for endomorphism fields constant in a right-invariant orthonormal frame.

\begin{rem}
The infinitesimal deformations described in Proposition~\ref{deforA} can be obtained as the tangent vectors to actual smooth curves in $\cCloc_{\SS^3}$, as described after Definition \ref{defordef}. This comes from the fact that the solutions we know form a differentiable manifold with four connected components.
\end{rem}

\section{Endomorphisms fields on the $3$-sphere and the second fundamental form} \label{metric}

\subsection{Thickening of the three-sphere} In the real-analytic case, it was shown \cite[Theorem 1.1]{A2M} that the existence of a Cauchy spinor $\psi$ over $(M,g)$ is equivalent to the existence of a metric of the form $g^\mathcal Z = \dd t^2 + g_t$ on a cylinder $(-\epsilon,\epsilon)\times M$ with $g_0 = g$, which carries a parallel spinor $\Psi$. In this case, the Cauchy endomorphism $A$ from \eqref{introgks} is the second fundamental form of the hypersurface $\{ 0 \} \times M$, and $\psi$ is the restriction of $\Psi$ to this hypersurface. Moreover, the germ near $t=0$ of such a metric is unique by analyticity.

When $A$ is plus or minus the identity on the sphere $\SS^3$, the resulting metric on $\cZ$ is flat and isometric to the induced metric on a tubular neighborhood of $\SS^3$ through the canonical embedding in $\RR^4$.

In this section we shall describe the ambient metric obtained by this thickening procedure in the case where the Cauchy endomorphism field is $A_0$ defined in \eqref{defA0} in terms of the left-invariant orthonormal frame $(e_1,e_2,e_3)$. 
Once again, a similar analysis can be carried out in the case of an endomorphism field constant in a \emph{right}-invariant orthonormal frame, as it amounts to reversing the orientation on $\cZ$.

For an interval $I \ni 0$ (to be defined later), we look for a metric on $\mathcal Z := \SS^3 \times I$ of the form
\begin{align} \label{metricgt}
g &:= \dd t^2 + g_t, &
g_t =: a(t)^2 e_1^2 + b(t)^2 (e_2^2 + e_3^2)
\end{align}
such that $\Sigma \mathcal Z$ carries a parallel spinor $\Psi$ and $A_0$ is the Weingarten map of $\SS^3 \times\{ 0 \}$. Moreover, $g_0$ has to coincide with the metric of the round sphere, so $a(0) = b(0) = 1$, and we assume that the functions $a$ and $b$ are non-negative.

We introduce the notation $M_t := \SS^3 \times \{ t \}$. Because of the form of the metric $g_t$ defined by \eqref{metricgt}, the hypersurfaces $M_t$ are Berger spheres.

The covariant derivative on $(M_t, g_t)$ is denoted $\nabla^t$. We also denote by $\R^t$ the Riemann curvature tensor of $(M_t, g_t)$, and by $A_t$ the Weingarten map of the hypersurface $M_t$ in $\mathcal Z$.

The fact that $\Psi$ is parallel in $\Sigma \mathcal Z$ implies that each hypersurface $M_t$ carries a Cauchy spinor (the restriction of $\Psi$ to $M_t$) with associated endomorphism field $A_t$. In particular, Equation \eqref{introendcond} gives for all $t$
\begin{equation} \label{endEqt}
0 = \R^t (X,Y) + *d^{\nabla^t} A_t (X, Y) + A_t (X) \wedge A_t (Y)
\end{equation}
for all $X,Y \in T M_t \simeq T M$. 
In turn, the identity \eqref{endEqt} for all $t$ implies the existence of a parallel spinor on $\cZ$ by the main result of \cite{A2M}.

We shall identify vectors and 1-forms on $M_t$  using the metric $g_t$. Notice that the Hopf frame $(e_1,e_2,e_3)$ is orthogonal, but not orthonormal on $M_t$. Using Koszul's formula and the expression \eqref{crej}
for the Lie brackets of $e_1, e_2, e_3$, we obtain by a straightforward computation that
\begin{align*}
\nabla^t_{e_1} e_2 &= \left( 2 - \frac{a^2}{b^2}(t) \right) e_3, & \nabla^t_{e_2} e_1 &= -\frac{a^2}{b^2}(t) e_3, &
\nabla^t_{e_1} e_3 &= \left( \frac{a^2}{b^2}(t) - 2 \right) e_2, \\
\nabla^t_{e_3} e_1 &= \frac{a^2}{b^2}(t) e_2, & \nabla^t_{e_2} e_3 &= e_1, & \nabla^t_{e_3} e_2 &= - e_1 \\
\nabla^t_{e_1} e_1 &= 0, & \nabla^t_{e_2} e_2 &= 0, & \nabla^t_{e_3} e_3 &= 0.
\end{align*}
From this we immediately get
\begin{align*}
\R^t_{e_1,e_2} e_1 &= - \frac{a^4}{b^4}(t) e_2,&
\R^t_{e_1,e_2} e_2 &= \frac{a^2}{b^2}(t) e_1,&
\R^t_{e_1,e_2} e_3 &= 0\\
\R^t_{e_1,e_3} e_1 &= - \frac{a^4}{b^4}(t) e_3,&
\R^t_{e_1,e_3} e_2 &= 0,&
\R^t_{e_1,e_3} e_3 &= \frac{a^2}{b^2}(t) e_1\\
\R^t_{e_2,e_3} e_1 &= 0,&
\R^t_{e_2,e_3} e_2 &= \left( 3 \frac{a^2}{b^2}(t) - 4 \right) e_3&
\R^t_{e_2,e_3} e_3 &= \left( 4 - 3 \frac{a^2}{b^2}(t) \right) e_3.
\end{align*}
Thus, in the basis $(e_1, e_2, e_3)$ one has 
\begin{align*}
\R^t_{e_1,e_2} = 
\begin{bmatrix}
0 & \frac{a^2}{b^2} & 0 \\
- \frac{a^4}{b^4} & 0 & 0 \\
0 & 0 & 0
\end{bmatrix}&&
\R^t_{e_1,e_3} = \begin{bmatrix}
0 & 0 & \frac{a^2}{b^2} \\
0 & 0 & 0 \\
- \frac{a^4}{b^4} & 0 & 0
\end{bmatrix}&&
\R^t_{e_2,e_3} = \begin{bmatrix}
0 & 0 & 0 \\
0 & 0 & 4 - 3 \frac{a^2}{b^2} \\
0 & 3 \frac{a^2}{b^2} - 4 & 0
\end{bmatrix}
\end{align*}
so, in terms of $2$-forms, these $t$-dependent curvature matrices become
\begin{align*}
\R^t (e_1,e_2) = - \frac{a^2}{b^4} e_1 \wedge e_2 && \R^t (e_1,e_3) = - \frac{a^2}{b^4} e_1 \wedge e_3 &&
\R^t (e_2,e_3) = \frac{3 a^2 - 4 b^2}{b^4} e_2 \wedge e_3.
\end{align*}

Let us now analyze the Weingarten maps $A_t$. By \cite[Proposition 4.1]{BGM}, $A_t$ is computed in the frame $\{e_1, e_2, e_3\}$ by the formula $g_t(A_t(X), Y) = -\frac{1}{2} \ddt \left(g_t(X,Y) \right)$. We obtain
\begin{equation*}
A_t = - \frac{\dot{a}}{a}(t) e_1 \otimes e_1 - \frac{\dot{b}}{b}(t) e_2 \otimes e_2 - \frac{\dot{b}}{b}(t) e_3 \otimes e_3.
\end{equation*}
The twisted exterior differential of $A_t$ is
\begin{align*}
(\dd^{\nabla^t}A_t)(e_1,e_2) &= \nabla^t_{e_1} (A_t e_2) - A_t \nabla^t_{e_1} e_2 - \nabla^t_{e_2} (A_t e_1) + A_t \nabla^t_{e_2} e_1 \\
&= - \frac{\dot{b}}{b}(t) \nabla^t_{e_1} e_2 - \left( 2 - \frac{a^2}{b^2}(t) \right) A_t e_3 + \frac{\dot{a}}{a}(t) \nabla^t_{e_2} e_1 - \frac{a^2}{b^2}(t) A_t e_3 \\
&= \frac{a^2}{b^2}(t) \left( \frac{\dot{b}}{b}(t) - \frac{\dot{a}}{a}(t) \right) e_3 \\
(\dd^{\nabla^t}A_t)(e_1,e_3) &= \nabla^t_{e_1} (A_t e_3) - A_t \nabla^t_{e_1} e_3 - \nabla^t_{e_3} (A_t e_1) + A_t \nabla^t_{e_3} e_1 \\
&= - \frac{\dot{b}}{b}(t) \nabla^t_{e_1} e_3 - \left( \frac{a^2}{b^2}(t) - 2 \right) A_t e_2 + \frac{\dot{a}}{a}(t) \nabla^t_{e_3} e_1 + \frac{a^2}{b^2}(t) A_t e_2 \\
&= \frac{a^2}{b^2}(t) \left( \frac{\dot{a}}{a}(t) - \frac{\dot{b}}{b}(t) \right) e_2 \\
(\dd^{\nabla^t}A_t)(e_2,e_3) &= \nabla^t_{e_2} (A_t e_3) - A_t \nabla^t_{e_2} e_3 - \nabla^t_{e_3} (A_t e_2) + A_t \nabla^t_{e_3} e_2 \\
&= - \frac{\dot{b}}{b}(t) \nabla^t_{e_2} e_3 - A_t e_1 + \frac{\dot{b}}{b}(t) \nabla^t_{e_3} e_2 - A_t e_1 \\
&= 2 \left( \frac{\dot{a}}{a}(t) - \frac{\dot{b}}{b}(t) \right) e_1.
\end{align*}
Thus
\begin{align*}
* (\dd^{\nabla^t}A_t)(e_1,e_2) &= \frac{a}{b^2}(t) \left( \frac{\dot{b}}{b}(t) - \frac{\dot{a}}{a}(t) \right) \,e_1 \wedge e_2, \\
*(\dd^{\nabla^t}A_t)(e_1,e_3) &= \frac{a}{b^2}(t) \left( \frac{\dot{b}}{b}(t) - \frac{\dot{a}}{a}(t) \right) \,e_1 \wedge e_3, \\
*(\dd^{\nabla^t}A_t)(e_2,e_3) &= \frac{2 a}{b^2}(t) \left( \frac{\dot{a}}{a}(t) - \frac{\dot{b}}{b}(t) \right) \,e_2 \wedge e_3.
\end{align*}
Finally, Equation \eqref{endEqt} rewrites as the system
\begin{equation*}
\left\lbrace
\begin{aligned}
0 &= \R^t_{e_1,e_2} + * (\dd^{\nabla^t} A_t) (e_1,e_2) + A_t(e_1) \wedge A_t(e_2) \\
0 &= \R^t_{e_1,e_3} + * (\dd^{\nabla^t} A_t) (e_1,e_3) + A_t(e_1) \wedge A_t(e_3) \\
0 &= \R^t_{e_2,e_3} + * (\dd^{\nabla^t} A_t) (e_2,e_3) + A_t(e_2) \wedge A_t(e_3)
\end{aligned}
\right.
\end{equation*}
and taking into account the previous computations, this system reads
\begin{equation*}
\left\lbrace
\begin{aligned}
0 &= - \frac{a^2}{b^4} e_1 \wedge e_2 + \frac{a}{b^2} \left( \frac{\dot{b}}{b} - \frac{\dot{a}}{a} \right) e_1 \wedge e_2 + \frac{\dot{a} \dot{b}}{ab} e_1 \wedge e_2 \\
0 &= - \frac{a^2}{b^4} e_1 \wedge e_3 + \frac{a}{b^2} \left( \frac{\dot{b}}{b} - \frac{\dot{a}}{a} \right) e_1 \wedge e_3 + \frac{\dot{a} \dot{b}}{ab} e_1 \wedge e_3 \\
0 &= \frac{3 a^2 - 4 b^2}{b^4} e_2 \wedge e_3 + \frac{2 a}{b^2} \left( \frac{\dot{a}}{a} - \frac{\dot{b}}{b} \right) e_2 \wedge e_3 + \left( \frac{\dot{b}}{b} \right)^2 e_2 \wedge e_3
\end{aligned}
\right.
\end{equation*}
so we are left with the two independent equations
\begin{equation} \label{system}
\left\lbrace
\begin{aligned}
0 &= - \frac{a^2}{b^4} - \frac{\dot{a}}{b^2} + \frac{a \dot{b}}{b^3} + \frac{\dot{a} \dot{b}}{ab} \\
0 &= \frac{3 a^2 - 4 b^2}{b^4} + \frac{2 \dot{a}}{b^2} - \frac{2 a \dot{b}}{b^3} + \left( \frac{\dot{b}}{b} \right)^2.
\end{aligned}
\right.
\end{equation}
Moreover, from the identity $g_0 (A_0(X), Y) = -\frac{1}{2} \dot{g}_0 (X, Y)$ for all $X,Y \in T M_0$, we have the initial conditions
\begin{align*}
a(0) = b(0) = 1&& \dot{a}(0) = - 1&& \dot{b}(0) = 3.
\end{align*}

The first equation of \eqref{system} can be rewritten as
\begin{align*}
0 & = \frac{1}{ab}\left( \dot{b} - \frac{a}{b} \right) \left( \frac{a^2}{b^2} + \dot{a} \right),
\end{align*}
thus either $\dot{b} - \frac{a}{b} = 0$ or $\frac{a^2}{b^2} + \dot{a} = 0$. However, from the initial conditions one has $\dot{b}(0) - \frac{a}{b}(0) =  2 \neq 0$ so the first case never occurs. Consequently, we get
\begin{equation*}
\dot{a} = - \frac{a^2}{b^2}.
\end{equation*}
Substituting this $\dot{a}$ in the second equation of the system \eqref{system}, one obtains by factorization
\begin{align*}
0 &= \frac{1}{b^4}(a - b \dot{b} - 2 b) (a - b \dot{b} + 2 b).
\end{align*}
Thus, either $a - b \dot{b} - 2 b = 0$ or $a - b \dot{b} + 2 b = 0$. The initial conditions give $a(0) - b \dot{b}(0) = 1 - 3 = - 2$
and we conclude that the second case occurs. We have reduced \eqref{system} to the simpler system
\[
\left\lbrace
\begin{aligned}
&\dot{a} = - \frac{a^2}{b^2} \\
&\dot{b} = \frac{a}{b} + 2 \\
&a(0) = b(0) = 1, \, \dot{a}(0) = -1, \, \dot{b}(0) = 3.
\end{aligned}
\right.
\]
In order to solve this system, we will find a conserved quantity and make a well-chosen change of variable. We begin by computing the derivative of $\frac{b}{a}$:
\[
\left( \frac{b}{a} \right)' = \frac{\dot{b} a - \dot{a} b}{a^2} = \frac{\left( \frac{a}{b} + 2 \right) a + \frac{a^2}{b^2}b}{a^2} = \frac{2}{b} + \frac{2}{a},
\]
and then,
\[
0 = \left( \frac{b}{a} \right)' - 2 \left( \frac{1}{b} + \frac{1}{a} \right) \Leftrightarrow 0 = a b \left( \frac{b}{a} \right)' - 2 (a + b).
\]
In addition, one has
\[
(a b)' = \dot{a} b + \dot{b} a = - \frac{a^2}{b^2} b + a \left( \frac{a}{b} + 2 \right) = 2 a,
\]
thus
\[
2 (a + b) = 2 a + 2 \frac{b}{a} a = 2 a \left( \frac{b}{a} + 1 \right) = (a b)' \left( \frac{b}{a} + 1 \right),
\]
and finally we have
\[
0 = a b \left( \frac{b}{a} \right)' - (a b)' \left( \frac{b}{a} + 1 \right) \Leftrightarrow \left( \frac{1}{a b} \left( \frac{b}{a} + 1 \right) \right)' = 0.
\]
We conclude that the quantity $\frac{1}{a b} \left( \frac{b}{a} + 1 \right)$ is constant, so
\begin{equation*}
\frac{1}{a b} \left( \frac{b}{a} + 1 \right) = 2.
\end{equation*}
A natural change of variable is to set $s = \phi(t) := a(t) b(t)$, so $\frac{b(t)}{a(t)} = 2 \phi(t) - 1= 2s-1$. Composing by $\phi^{-1}$ on the right, we obtain for $s$ in a neighborhood of $1$
\begin{align*}
&\dot{a} (\phi^{-1} (s)) = - \frac{a^2}{b^2} (\phi^{-1} (s)) \\
&\dot{b} (\phi^{-1} (s)) = \frac{a}{b} (\phi^{-1} (s)) + 2 \\
&a(\phi^{-1} (s)) b(\phi^{-1} (s)) =  \phi (\phi^{-1} (s)) \\
&\frac{b}{a}(\phi^{-1} (s)) = 2 \phi(\phi^{-1} (s)) - 1
\end{align*}
and setting $\alpha := a \circ \phi^{-1}$ and $\beta := b \circ \phi^{-1}$ one arrives at
\begin{align*}
&\dot{\alpha} = - (\phi^{-1})' \frac{\alpha^2}{\beta^2} = - (\phi^{-1})' \frac{1}{(2 s -1)^2} \\
&\dot{\beta} = (\phi^{-1})' \left( \frac{\alpha}{\beta} + 2 \right) = (\phi^{-1})'  \frac{4 s - 1}{2 s -1} \\
&\alpha \beta = s \\
&\frac{\beta}{\alpha} = 2 s - 1.
\end{align*}
Differentiating the last two equations one gets
\begin{equation} \label{eq1}
\dot{\alpha} \beta + \alpha \dot{\beta} = 1
\end{equation}
and
\begin{align*}
\alpha \dot{\beta} - \dot{\alpha} \beta = 2 \alpha^2 &\Leftrightarrow 2 \alpha = \dot{\beta} - \dot{\alpha} (2 s - 1) = (\phi^{-1})'  \frac{4 s - 1}{2 s -1} + (\phi^{-1})' \frac{1}{2 s -1} \\
&\Leftrightarrow  \alpha = (\phi^{-1})' \frac{2 s}{2 s -1}.
\end{align*}
Re-injecting this last equation in \eqref{eq1} we have
\begin{align*}
& \alpha \left( \dot{\alpha} \frac{\beta}{\alpha} + \dot{\beta} \right) = 1 \Leftrightarrow (\phi^{-1})' \frac{2 s}{2 s -1} \left( - (\phi^{-1})' \frac{1}{(2 s -1)^2} (2 s - 1)  + (\phi^{-1})'  \frac{4 s - 1}{2 s -1} \right) = 1 \\
\Leftrightarrow & \left((\phi^{-1})' \right)^2 \frac{4 s}{2 s -1} = 1 \Leftrightarrow (\phi^{-1})' = \sqrt{\frac{2 s -1}{4 s}},
\end{align*}
where we used that the derivative of $\phi$ is positive.
Thus, one has
\begin{align*}
\dot{\alpha} &= - (\phi^{-1})' \frac{1}{(2 s -1)^2} = - \frac{1}{2 \sqrt{s} (2 s -1)^\frac{3}{2}} \\
\dot{\beta} &= (\phi^{-1})'  \frac{4 s - 1}{2 s -1} = \frac{4 s - 1}{2 \sqrt{2 s^2 - s}}
\end{align*}
and by integration we finally obtain
\begin{align*}
\alpha = \sqrt{\frac{s}{2s - 1}}, && \beta = \sqrt{s (2 s - 1)}.
\end{align*}

We conclude that in terms of the new variable $s$, the metric on $S^3$ introduced in \eqref{metricgt} is
\begin{align*}
g_s = \frac{2 s -1}{4 s} \dd s^2 + \frac{s}{2s - 1} \eta_1^2 + s (2 s - 1) (\eta_2^2 + \eta_3^2),&&s \in (\tfrac{1}{2}, \infty).
\end{align*}
We can compute the interval to which the variable $t$ belongs by calculating
\begin{align*}
\int_\frac{1}{2}^1 (\phi^{-1})' (s) \dd s &= \int_\frac{1}{2}^1 \sqrt\frac{2s - 1}{4s} \dd s \\
&= \frac{1}{\sqrt{2}} \left[s \sqrt{1 - \frac{1}{2s}} - \frac{1}{4} \ln \left( 4s - 1 + \sqrt{16 s^2 - 8 s} \right) \right]_{\frac{1}{2}}^1 \\
&= \frac{1}{2} - \frac{1}{4 \sqrt{2}} \ln (3 + 2 \sqrt{2}) = \frac{\sqrt{2} -\ln (1 + \sqrt{2}) }{2\sqrt{2}} \approx 0,1884,
\end{align*}
and so the metric \eqref{metricgt} admitting a parallel spinor exists for $t \in \left( \frac{1}{2} \left( \frac{1}{\sqrt{2}} \ln (1 + \sqrt{2}) - 1 \right), \infty \right)$.

\subsection{Link with the family of Euclidean Taub-NUT metrics}
Let us first extend the previous study to the case where the initial sphere has radius $r > 0$. It is easy to see that this does not change the form of the endomorphism $A_0$ defined in \eqref{defA0}. This modification results in the rescaling of the metric by a factor $r^2$. Subsequently, if we keep the same notations as in the case $r = 1$, the metric $g$ on $\cZ$ is given by
\begin{align} \label{eq40}
g = \frac{2 s -1}{4 s} r^2 \dd s^2 + \frac{r^2 s}{2s - 1} \eta_1^2 + r^2 s (2 s - 1) (\eta_2^2 + \eta_3^2),&& s\in (\tfrac{1}{2},\infty).
\end{align}

With the change of variable $u := r s$ we can express this metric by
\begin{align*}
g = \frac{2 u - r}{4 u} \dd u^2 + \frac{r^2 u}{2 u - r} \eta_1^2 + u (2 u - r) (\eta_2^2 + \eta_3^2),&& u\in(\tfrac{r}{2}, \infty).
\end{align*}

This family of metrics is strikingly similar to the well-known family of Euclidean Taub-NUT metrics on $\RR^4$ (see e.g.\ \cite{MM} and the references therein). 
In polar coordinates, the Euclidean Taub-NUT metrics are given (up to a constant) by the expression
\begin{equation}\label{gTN}
g_{TN} = \frac{a s + b}{s} \left( \dd s^2 + \frac{4 b^2 s^2}{(a s + b)^2} \eta_1^2 + 4 s^2 (\eta_2^2 + \eta_3^2) \right)
\end{equation}
where $a$ and $b$ are \emph{positive} parameters and $\eta_1,\eta_2,\eta_3$ are the $1$-forms dual to the Hopf vector fields $e_1,e_2,e_3$. 
Through a change of variable in the radial variable $s$, we can always normalize the parameter $a$ to be equal to $2$.

The metric we found in Equation \eqref{eq40} belongs formally to the extension for \emph{negative} values of the parameter $b = - r<0$ of the family of Taub-NUT metrics normalized with $a = 2$. Note that the parameter $r$ cannot vanish, else the metric degenerates. The metrics in this family admit a nonzero parallel spinor. This implies that they are hyperk\"ahler, hence Ricci-flat, and also (anti-) self-dual, according to the chirality of the nonzero parallel spinor.

In the presentation \eqref{gTN} of the Taub-NUT metric, $s=0$ is an apparent singularity, but in fact the metric extends smoothly in the origin of $\RR^4$. In contrast, the metric \eqref{eq40} has a true singularity at $s=\frac{1}{2}$. The horizontal directions $e_2$ and $e_3$ collapse, while the vertical direction of $e_1$ \emph{explodes} in finite time as $s\searrow\frac{1}{2}$. As a result, the curvature operator is unbounded near $s=\frac{1}{2}$.

\section{Classification results on $\SS^3$}

In this section we analyze the set of symmetric solutions of \eqref{introendcond} in the case $M = \SS^3$. Recall that $\cC_{\SS^3}=\cCloc_{\SS^3}$ because the sphere is simply-connected. Since the known examples of solutions can be expressed as constant matrices in a frame of left (or right-) invariant vector fields, we will investigate some classes of endomorphisms related to these vector fields.

\subsection{Endomorphisms constant in a left or right invariant orthonormal frame} \label{6.1}
The four examples we recalled in Remark~\ref{exEnd} can be interpreted as constant matrices either in a left- or a right-invariant orthonormal frame. For this reason, it is legitimate to search for all symmetric endomorphisms in $\cC_{\SS^3}$ that verify this property. We shall prove that there exist no other examples besides the ones already known from Remark~\ref{exEnd}.

\begin{prop} \label{constantA}
Let $A \in \cC_{\SS^3}$. Assume that $A$ is constant in a left (resp. right)-invariant orthonormal frame. Then, either $A = \pm \mathrm{Id}$ or $A$ has eigenvalues $1$, $-3$, $-3$ (resp. $-1$, $3$, $3$). In particular, $A$ is one of the endomorphism fields described in Remark~\ref{exEnd}.
\end{prop}

\begin{proof} We recall that $(e_1, e_2, e_3)$ are the three left-invariant vector fields corresponding to the quaternions $i,j,k$ on $\SS^3$.

Let $A \in \cC_{\SS^3}$ (i.e., $A$ is symmetric and satisfies Equation \eqref{endsphere} on $\SS^3$), and assume that $A$ is constant in a left-invariant orthonormal frame. Hence, $A$ can be viewed as a real symmetric $3 \times 3$ matrix, in particular it is diagonalizable. From these considerations, up to an isometry of the sphere we can assume without loss of generality that $A = a e_1 \otimes e_1 + b e_2 \otimes e_2 + c e_3 \otimes e_3$ for $a,b,c \in \RR$.

Equation \eqref{endsphere} applied to $X,Y \in \{e_1, e_2, e_3 \}$ together with \eqref{extcovder} and \eqref{covE} give the system
\begin{equation*}
\left\lbrace
\begin{aligned}
& a + b - 2 c = 1 - a b \\
& b + c - 2 a = 1 - b c \\
& c + a - 2 b = 1 - c a
\end{aligned}
\right.
\Leftrightarrow
\left\lbrace
\begin{aligned}
& (a+1) (b+1) = 2(c+1) \\
& (b+1) (c+1) = 2(a+1) \\
& (a+1) (c+1) = 2(b+1).
\end{aligned}
\right.
\end{equation*}
We easily see that $a+1 = 0 \Leftrightarrow b+1 = 0 \Leftrightarrow c+1 = 0$, and in this case $A = - \mathrm{Id}$.
We assume now that $a+1 \neq 0$. The product of all the equations give
\[
(a+1) (b+1) (c+1) = 8,
\]
and we conclude that $(a+1)^2 = (b+1)^2 = (c+1)^2 = 4$. Therefore $a+1,b+1,c+1\in\{-2,2\}$ and moreover an even number among them are negative, concluding the proof when $A$ is constant in a left-invariant frame.

The case of a right-invariant orthonormal frame is treated similarly and produces the additional solution $-1$, $3$, $3$.
\end{proof}

\subsection{Endomorphism fields with three distinct constant eigenvalues} The case of an endomorphism solution of \eqref{endsphere} with at most two distinct eigenvalues was already studied in \cite{MS3}, where it was shown that the only possibilities are the ones given in Remark~\ref{exEnd}. 

\begin{prop} \label{classification2}
There is no element of $\cC_{\SS^3}$ with three distinct constant eigenvalues.
\end{prop} 
\begin{proof}
Let $A \in \cC_{\SS^3}$ with three constant eigenvalues $\lambda_1 < \lambda_2 < \lambda_3$. The associated unitary eigenvectors are global vector fields on $\SS^3$, which form an orthonormal frame, and are denoted by $X_1, X_2, X_3$. Equation~\eqref{endsphere} means that for every cyclic permutation $(a,b,c)$ of the index set $(1,2,3)$ one has
\begin{equation*}
\lambda_b \nabla_{X_a} X_b - \lambda_a \nabla_{X_b} X_a - A [X_a, X_b] = (1 - \lambda_a \lambda_b)X_c .
\end{equation*}
Projecting on $X_a$, we see that
\begin{align*}
&\lambda_b g( \nabla_{X_a} X_b, X_a) - \lambda_a g (\nabla_{X_b} X_a, X_a) - \lambda_a g([X_a, X_b], X_a) = 0 \\
\Leftrightarrow & (\lambda_b - \lambda_a) g(\nabla_{X_a} X_b, X_a) = 0 \\
\Leftrightarrow & (\lambda_b - \lambda_a) g([X_a,X_b], X_a) = 0.
\end{align*}
This last equation is true for any $a, b \in \{1,2,3 \}$, and this means $[X_a,X_b] \in \Span (X_c)$.

As a direct consequence of Koszul formula, $\nabla_{X_a} X_a = 0$ for $a \in \{1,2,3 \}$. We can compute for any $a$
\[
\delta (X_a) = - X_k \lrcorner \nabla_{X_k} X_a = g( X_a, \nabla_{X_k} X_k) = 0.
\]
This shows that the vector fields $X_k$ are geodesic and divergence free. By a result of Gluck and Gu \cite[Theorem A]{GG}, every geodesic and divergence free vector field on $\SS^3$ is a Hopf vector field (i.e. a unit vector field tangent to the fiber of a Hopf fibration). Moreover, since they form an orthonormal basis at any point, they are all either left or right-invariant. However we can give a simpler argument in our case:

\begin{lemma}
Let $(X_1, X_2, X_3)$ be a global orthonormal frame of geodesic vector fields on $\SS^3$ (i.e. $\nabla_{X_k} X_k = 0$). Then $(X_1, X_2, X_3)$ is 
either a left- or a right-invariant frame.
\end{lemma}
\begin{proof}
Using Koszul's formula, one sees that the assumption $\nabla_{X_k} X_k = 0$ is actually equivalent to the existence of three real functions $\alpha_1, \alpha_2, \alpha_3$ on $\SS^3$ such that $[X_a, X_b] = \alpha_c X_c$ for any cyclic permutation $(a,b,c)$ of the indices $1,2,3$.

We define the real-valued functions $\beta_k = (-1)^{\delta_1^k} \alpha_1 + (-1)^{\delta_2^k} \alpha_2 + (-1)^{\delta_3^k} \alpha_3$. By the Koszul formula, 
\begin{align*}
\nabla_{X_a} X_b = \beta_a X_c, && \nabla_{X_a} X_c = - \beta_a X_b,
\end{align*}
Moreover, the curvature tensor on the sphere satisfies $\R (X_a, X_b) X_b = X_a$. Since the vector fields $X_1,X_2,X_3$ are geodesic, we also have
\begin{align*}
\R (X_a, X_b) X_b &= -\nabla_{X_b} \nabla_{X_a} X_b - \nabla_{[X_a, X_b]} X_b \\
&= - X_b (\beta_a) X_b - \beta_a \beta_b X_a + 2 \alpha_c \beta_c X_a \\
&= - X_b (\beta_a) X_b + (- \beta_a \beta_b + \beta_a \beta_c + \beta_b \beta_c) X_a,
\end{align*}
so the projection of this equation on $X_a$ yields
\[
- \beta_a \beta_b + \beta_a \beta_c + \beta_b \beta_c =1.
\]
Since this is true for any value of $(a, b, c)$ in $\{ (1,2,3), (2,3,1), (3,1,2) \}$, one has $\beta_a \beta_b = 1$ and we conclude that $\beta_k = \pm 1$ for any $k \in \{1,2,3 \}$.

Assume first that $\beta_k = 1$ for any $k \in \{1,2,3 \}$. We define for any $X,Y \in T \SS^3$ the covariant derivative
\begin{equation*}
\overline \nabla_X Y := \nabla_X Y - * (X \wedge Y),
\end{equation*}
which was already considered in \eqref{modifiedconnection}. The vector fields $X_k$ are parallel for $\overline \nabla$ and so are the left-invariant vector fields with value $(X_k)_e$ at $e$. Thus, these vector fields coincide.

In the case $\beta_1=\beta_2=\beta_3 = -1$, the same proof shows that $X_1, X_2, X_3$ are right-invariant.
\end{proof}

Hence, $A$ is constant in a left or right-orthonormal frame and according to Proposition~\eqref{constantA} it must have at most 2 different eigenvalues, which contradicts the hypothesis.
\end{proof}

\subsection{Endomorphisms constant in the direction of a left-invariant vector field} We will now weaken the condition from Section \ref{6.1}, and search for solutions $A$ of \eqref{endsphere} on $\SS^3$ that are constant in the direction of a left-invariant vector field $\xi$, i.e. $\mathcal L_\xi A = 0$. Assuming this invariance, all the objects can be expressed on the basis of the Hopf fibration with fibers tangent to $\xi$.
We decompose $A$ under the form:
\begin{equation}\label{Ades}
A = f \xi \otimes \xi + v \otimes \xi + \xi \otimes v + B
\end{equation}
where $f$ is a function on $\SS^3$, $v \in \xi^\perp$ and $B$ is the restriction of $A$ to $\xi^\perp$. The condition $\mathcal L_{\xi} A = 0$ gives
\begin{align*}
0 = \mathcal L_{\xi} A = (\xi f) \xi \otimes \xi + \mathcal L_{\xi} v \otimes \xi + \xi \otimes \mathcal L_{\xi} v + \mathcal L_{\xi} B,
\end{align*}
and we know that for all $X \in \xi^\perp$, $\mathcal L_{\xi} X \in \xi^\perp$ since $\xi$ is a Killing field, so we deduce
\begin{align} \label{liesys}
\xi f = 0, && \mathcal L_{\xi} v = 0, && \mathcal L_{\xi} B = 0.
\end{align}
As a direct consequence of equations \eqref{liesys} we can interpret $f$, $v$ and $B$ respectively as a function, a vector and an endomorphism on the basis $\SS^2(\frac 1 2)$ of the Hopf fibration.

We define the endomorphism $J $ of $\xi^\perp$ by $J X := - \nabla_X \xi$. This endomorphism is skew-symmetric and satisfies $J^2 = -1$; it is actually the lift of the standard almost complex structure from $\SS^2(\frac{1}{2})$ through the Hopf fibration, so we will see it as an endomorphism of the base.

The invariance equations \eqref{liesys} give
\begin{equation*}
\begin{aligned}
\nabla_{\xi} v &= \nabla_v \xi = - J v \\
(\nabla_{\xi} B) X &= \nabla_{\xi} (B X) - B \nabla_{\xi} X  = \nabla_{B X} \xi + [\xi, B X] + B J X - B [\xi, X] \\
&= [B,J] X + \mathcal L_{\xi} (B X) - B \mathcal L_{\xi} X = [B,J] X + (\mathcal L_{\xi} B) X = [B,J] X.
\end{aligned}
\end{equation*}
Now, we express Equation \eqref{endsphere} in terms of $f$, $v$ and $B$ by considering horizontal and vertical vectors for $X$ and $Y$.

Let $X,Y$ be two orthogonal vector fields in $\xi^\perp$. One has,
\begin{align*}
d^\nabla A(\xi,X) &= (\nabla_{\xi} A) X - (\nabla_X A) \xi \\
&= - (J v \otimes \xi + \xi \otimes J v - [B,J]) X - \nabla_X (A \xi) - A J X \\
&= - g(J v, X) \xi - J B X + B J X - \nabla_X (f \xi + v) - g (v , J X) \xi - B J X \\
&= - J B X - \nabla_X (f \xi + v).
\end{align*}
We now use the fact that for any $X$, $g(J X, X) = 0$ to infer that $* \xi \wedge X = J X$ and $* v \wedge X = - g( v, J X) \xi$, so Equation \eqref{endsphere} implies
\begin{align*}
- J B X - \nabla_X (f \xi + v) &= * \xi \wedge X - * (f \xi + v) \wedge (g(X,v) \xi + B X) \\
&= J X - f J B X + g (v, J B X) \xi + g (X, v) J v.
\end{align*}
Projecting this equation on $\xi$, one has
\[
- g (\nabla_X (f \xi + v), \xi) = g(v, J B X)
\]
thus
\[
g (B J v + J v - d f, X) = 0.
\]
Since this last equation is true for any $X \in \xi^\perp$ we conclude
\begin{equation} \label{eq47}
B J v = - J v + d f.
\end{equation}
We define the orthogonal projector $P$ on $\xi^\perp$. We now project Equation \eqref{eq47} on the orthogonal of $\xi$, and we obtain
\[
- J B X + f J X - P \nabla_X v = J X - f J B X + g (X, v) J v \]
thus
\[
(f - 1)(J B X + J X) = P \nabla_X v + g (X, v) J v.
\]

Thus, we have the system
\begin{equation*}
\left\lbrace
\begin{aligned}
&B J v = - J v + d f \\
&(f - 1)(J B X + J X) = P \nabla_X v + g (X, v) J v
\end{aligned}
\right.
\end{equation*}

We now compute:
\begin{align*}
(\nabla_X A) Y =& (X(f) \xi \otimes \xi - f J X \otimes \xi - f \xi \otimes J X + \nabla_X v \otimes  \xi \\
&+ \xi \otimes \nabla_X v - J X \otimes v - v \otimes J X + (\nabla_X B)) Y \\
=& - f g (J X, Y) \xi + g (\nabla_X v, Y) \xi - g(J X, Y) v \\
&- g(v, Y) J X + (\nabla_X B) Y.
\end{align*}
We deduce that
\begin{align*}
d^\nabla A (X, Y) =& - f g (J X, Y) \xi + g (\nabla_X v, Y) \xi - g(J X, Y) v - g(v, Y) J X + (\nabla_X B) Y \\
&+ f g (J Y, X) \xi - g (\nabla_Y v, X) \xi + g(J Y, X) v + g(v, X) J Y - (\nabla_Y B) X \\
=& 2 f g (X, J Y) \xi + d v (X,Y) \xi + 2 g(X, J Y) v \\
&- g(v, Y) J X + g(v, X) J Y + d^\nabla B (X,Y).
\end{align*}
Equation \eqref{endsphere} leads to
\begin{align*}
&2 f g (X, J Y) \xi + d v (X,Y) \xi + 2 g(X, J Y) v - g(v, Y) J X + g(v, X) J Y + d^\nabla B (X,Y) \\
=& * X \wedge Y - * (g(X,v) \xi + B X) \wedge (g(Y,v) \xi + B Y) \\
=& - [g (X,J Y) - g (B X, J B Y)] \xi - g(X,v) J B Y + g(Y,v) J B X
\end{align*}
which leads to
\begin{align*}
[- g (X,J Y) + g (B X, J B Y) - 2 f g (X, J Y) - d v (X,Y)] \xi - 2 g(X, J Y) v \\
= d^\nabla B (X,Y) + g(X,v) J (B + 1) Y - g(Y,v) J (B + 1) X.
\end{align*}
Since
\begin{align*}
g (d^\nabla B (X,Y), \xi) &= g (\nabla_X (B Y) - \nabla_Y (B X) - B [X, Y], \xi) \\
&= g (\nabla_X (B Y), \xi) - g(\nabla_Y (B X), \xi) \\
&= - g (B X, J Y) + g(B Y, J X),
\end{align*}
the projection on $\xi$ gives
\[
- g (X,J Y) + g (B X, J B Y) - 2 f g (X, J Y) - d v (X,Y) = - g (B X, J Y) + g(B Y, J X)
\]
and
\[
2 (1+ f) g (X, J Y) + d v (X,Y) = g ((B+1) X, J (B+1) Y).
\]
We remark that for any symmetric endomorphism $S$, one has $S J S = \det (S) J$ and the last equation is rewritten
\[
[2 (1 + f) - \det (B + 1)] g (X, J Y) = - d v (X,Y).
\]
Moreover, the projection on $\xi^\perp$ provides the equation
\begin{align*}
- P d^\nabla B (X,Y) = g(X,v) J (B + 1) Y - g(Y,v) J (B + 1) X + 2 g(X, J Y) v.
\end{align*}
For the remainder of this section, we will denote by $\overline \nabla$ the covariant derivative on $\SS^2 (\frac{1}{2})$ (the basis of the Hopf fibration). We recall that if $U,V$ are basic vector fields on $\SS^3$ we have the equation
\begin{equation*}
\nabla_U V = \overline \nabla_U V + g (V, J U) \xi.
\end{equation*}
Now, we study the objects on $\SS^2(\frac{1}{2})$. For any $X,Y \in T \SS^2(\frac{1}{2})$, one has $d v (X,Y) = d^* (J v) g(J X, Y)$ and, using a unit vector field $X$ at a point of $\SS^2(\frac{1}{2})$ and the fact that $J$ is parallel, one has
\[
* d^{\overline \nabla} B = d^{\overline \nabla} B (X,J X) = (\overline \nabla_X B) J X - (\overline \nabla_{J X} B) X = (\overline \nabla_X B J) X + (\overline \nabla_{J X} B J) J X = - \delta^{\overline \nabla} (B J).
\]
Consequently, we obtain four equations on the sphere $\SS^2(\frac{1}{2})$
\begin{equation}\label{sd}
\left\lbrace
\begin{aligned}
&(B +1 ) J v = d f \\
&(f - 1) J (B + 1) = \overline \nabla v + v \otimes J v \\
&2 (1 + f) - \det (B + 1) = d^* (J v) \\
&\delta^{\overline \nabla} (B J) = J (B + 3) J v.
\end{aligned}
\right.
\end{equation}

The system \eqref{sd} seems too difficult to solve in full generality for the time being, and is left as an open problem. 
\subsection{A particular case: $v=0$}
We proceed by treating only the special case in the system \eqref{sd} where $\xi$ is an eigenvector of $A$, i.e.\ $v = 0$ in \eqref{Ades}. In this situation, the system reduces to
\begin{equation} \label{sysend}
\left\lbrace
\begin{aligned}
&(f - 1) J (B + 1) X = 0 \\
&[2 (1 + f) - \det (B + 1)] = 0 \\
&\delta^{\overline \nabla} (B J) = 0.
\end{aligned}
\right.
\end{equation}

\begin{prop} \label{classification1}
Let $A \in \cC_{\SS^3}$ be the symmetric endomorphism corresponding to a Cauchy spinor on $\SS^3$. Assume that there exists a left (resp. right)-invariant vector field $\xi$ such that $\mathcal L_\xi A = 0$ and such that $\xi$ is an eigenvector of $A$. Then $A = \pm \mathrm{Id}$ or $A = \xi \otimes \xi - 3 P$ (resp. $A = - \xi \otimes \xi + 3 P$), where $P$ is the orthogonal projector on $\xi^\perp$.
\end{prop}
\begin{proof}
We start with the case where $\xi$ is left-invariant. We have seen that $A$ must equal $f\xi\otimes\xi+B$ where the Killing vector field $\xi$ is an eigenvector of $A$, and $f$ and $B$ are $\xi$-invariant, hence they are pull-back of objects from the base $\SS^2(\frac{1}{2})$ of the Hopf fibration.
Define the open set $\cO=\lbrace x \in \SS^2, \, B(x) \neq -\Id \rbrace\subset \SS^2$. On $\mathcal O$, the first equation gives $f = 1$, and the second one gives $\det (B + 1) = 4$. By continuity, this is still true on $\overline{\mathcal O}$, hence $\overline\cO\subset\cO$. We conclude that $\cO$ is open and closed, and by connectedness it is either empty or equal to $\SS^2$.

If $\cO$ is empty, $B = - \mathrm{Id}$ on $\SS^2$, so $f=-1$ and the only solution of the system is $A = -\Id$. 

If $\cO=\SS^2$, we have seen that $f=1$ so we are left with the two equations $\det (B + 1) = 4$ and $\delta^{\overline \nabla} (B J) = 0$. Equivalently, we search for a symmetric endomorphism field $C=\frac{B+1}{2}$ on the unit sphere $\SS^2$ with $\det C = 1$ and $\delta^{\overline \nabla} (C J) =0$. 
Notice that $\delta^{\overline \nabla} (JCJ) = - J \delta^{\overline \nabla} (C J) = 0$, hence the symmetric endomorphism $U := JCJ$ satisfies 
$\det U = 1$ and $\delta^{\overline \nabla} U = 0$. 
The following proposition contains the result we need:

\begin{prop} \label{s2}
Let $U$ be a symmetric endomorphism on $\SS^2$ which satisfies $\det U = 1$ and $\delta^{\overline \nabla} U = 0$. Then $U = \pm \mathrm{Id}$.
\end{prop}
\begin{proof}
Let $t := \frac{1}{2} \mathrm{tr} U$, and $S := J (U - t \mathrm{Id})$. Since $U - t \mathrm{Id}$ is symmetric and traceless, so is $S$. We will use several times below that traceless symmetric endomorphisms anticommute with $J$. One has $\delta^{\overline \nabla} S = \delta^{\overline \nabla} (J U) - \delta^{\overline \nabla} (J t \mathrm{Id}) = J d t$. Since $U^2 - 2 t U + \det U = 0$ by Cayley-Hamilton's theorem, one has
\begin{equation} \label{eq555}
S^2 = (J S)^2 = (U - t Id)^2 = U^2 - 2 t U + t^2 \mathrm{Id} = (t^2 - 1) \mathrm{Id}.
\end{equation}
Thus, $\vert t \vert \ge 1$ because $S^2=SS^*$ is non-negative. The function $t$ is continuous on $\SS^2$, so it cannot change sign. It follows that either $t\geq 1$ on $\SS^2$ or $t\leq -1$ on $\SS^2$. We shall show below that in the first case $U=\Id$. In the second case, $-U$ also satisfies the hypotheses of the proposition and moreover $\tr(-U)>0$, so by the first case we get $-U=\Id$. It suffices therefore to solve the first case, i.e., we can assume in the rest of the proof that $t\geq 1$ on $\SS^2$.

We want to show that the open set $E := \lbrace x \in \SS^2; t(x) > 1 \rbrace$ is empty. One has $S \neq 0$ on $E$, thus we can define $T := \frac{S}{\Vert S \Vert}$. Since $S$ is symmetric and traceless, so are $T$ and $JT$, hence $(T, J T)$ is an orthonormal frame over $E$ in the bundle of symmetric traceless endomorphisms of $T\SS^2$. Therefore, 
\begin{equation} \label{defconnectionform}
\overline \nabla T =: \alpha \otimes J T
\end{equation}
for some $1$-form $\alpha \in \Lambda^1 (E)$. By taking a further covariant derivative in this equation and skew-symmetrizing one has for any $X, Y \in T \SS^2$
\[
d \alpha (X,Y) J T = \bar {\R}_{X,Y} T
\]
where $\bar \R$ is the curvature of $\SS^2$, which acts on any endomorphism field $W$ as $\bar \R_{X,Y} W = [\bar \R_{X,Y}, W]$. Then, using the identity $\bar \R_{X,Y} = - X \wedge Y$  on $\SS^2$, we obtain $d \alpha = - 2 vol$, where $vol$ is the Riemannian volume form.

We identify $\alpha$ with a vector field $\alpha^\sharp$ via the metric, and one has $\delta^{\overline \nabla} T = - J T \alpha^\sharp$ by Equation \eqref{defconnectionform}. Since $T^2 = (J T)^2 = \frac{1}{2} \mathrm{Id}$ we obtain $\alpha^\sharp = - 2 J T (\delta^{\overline \nabla} T)$. The condition $\delta^{\overline \nabla} S = J \dd t$ gives $\delta^{\overline \nabla} T = \frac{\delta^{\overline \nabla} S}{\Vert S \Vert} + T \grad \ln (\Vert S \Vert) = \frac{J dt}{\Vert S \Vert} + T \grad \ln (\Vert S \Vert)$, hence
\begin{equation*}
\alpha^\sharp = - \frac{2 S (\grad (t))}{\Vert S \Vert^2} - J \grad \ln (\Vert S \Vert).
\end{equation*}
We now differentiate this equation to obtain
\[
d^* J \alpha^\sharp = - d^* \left(\frac{2 J S (\grad (t))}{\Vert S \Vert^2} \right) + \Delta \ln (\Vert S \Vert)
\]
where $\Delta = d^* d$ is the positive Laplacian on $\SS^2$. Using the fact that $d \alpha  = - 2 V$, which is equivalent to $d^* J \alpha^\sharp = -2$, one has
\[
- 2 = - d^* \left(\frac{2 J S (\grad (t))}{\Vert S \Vert^2} \right) + \Delta \ln (\Vert S \Vert).
\]
We know from \eqref{eq555} that $\Vert S \Vert^2 = 2 (t^2 - 1)$, so $t  = \sqrt {\frac{1}{2}\Vert S \Vert^2 + 1}$. This leads to $d t = \frac{\Vert S \Vert d (\Vert S \Vert)}{2 \sqrt {\frac{1}{2}\Vert S \Vert^2 + 1}}$, and finally
\begin{equation*}
- 2 = - d^* \left(\frac{J S d \ln (\Vert S \Vert)}{\sqrt {\frac{1}{2}\Vert S \Vert^2 + 1}} \right) + \Delta \ln (\Vert S \Vert).
\end{equation*}

We define $\zeta := \frac{1}{2} \Vert S \Vert^2$ and we rewrite the above equation as
\begin{equation*}
- 1 = - d^* (\frac{J S d \ln \zeta}{\sqrt {\zeta + 1}}) + \Delta \ln \zeta.
\end{equation*}
Let $x \in \SS^2$ be a point where $t $, and thus $\Vert S \Vert$, reaches its maximum. Clearly, $x \in E$. Let $(e_1, e_2)$ be an orthonormal frame which is parallel at $x$. At this point, one has
\begin{align*}
d^* (\frac{J S d \ln \zeta}{\sqrt{\zeta + 1}}) &= - \langle J \overline \nabla_{e_1} (\frac{S d \ln \zeta}{\sqrt {\zeta + 1}}), e_1 \rangle - \langle J \overline \nabla_{e_2} (\frac{S d \ln \zeta}{\sqrt {\zeta + 1}}), e_2 \rangle \\
&= \frac{\left[ \langle \overline \nabla_{e_1} d \ln \zeta, S J e_1 \rangle + \langle \overline \nabla_{e_2} d \ln \zeta , S J e_2 \rangle \right]}{\sqrt {\zeta + 1}} \\
&= \frac{\langle \Hess \ln \zeta , S J \rangle}{\sqrt {\zeta + 1}},
\end{align*}
where the Hessian is defined for any function $\beta$ by $\Hess \beta := \overline \nabla d \beta$.
Thus,
\begin{align*}
- 1 &= - \frac{\langle \Hess \ln \zeta , S J \rangle}{\sqrt {\zeta + 1}} - \langle \Hess \ln \zeta, \mathrm{Id} \rangle \\
&= - \langle \Hess \ln \zeta, \mathrm{Id} + M \rangle,
\end{align*}
where $M := \frac{S J}{\sqrt {\zeta + 1}}$. Since $\zeta = \frac{1}{2} \Vert S \Vert^2$ one has
\begin{align*}
\Vert M \Vert =  \sqrt{\frac{2 \zeta}{\zeta + 1}} < \sqrt 2.
\end{align*}
Since the trace of $M$ vanishes, $\Vert M \Vert = \sqrt 2 \rho(M)$, where $\rho(M)$ is the spectral radius of $M$. Thus one has $\rho(M) < 1$, so $\mathrm{Id} + M$ is positive definite. We now use the following elementary result:

\begin{lemma}
Let $N_1, N_2 \in \mathcal S_n (\RR)$ be two symmetric matrices such that $N_1$ is positive and $N_2$ is non-positive. Then, $\langle N_1 N_2 x, x \rangle \le 0$ for all $x \in \RR^n$.
\end{lemma}
\begin{proof}[Proof of the lemma]
By working in a eigenbasis $(f_1, \ldots, f_n)$ of $N_1$, we can suppose that $N_1$ is diagonal. Thus, for any $j \in \{1, \ldots, n \}$ we have $N_1 f_j =: \lambda_j f_j$ and
\[
\langle N_1 N_2 f_j, f_j \rangle = \langle N_2 f_j, N_1 f_j \rangle = \lambda_j \langle N_2 f_j, f_j \rangle \le 0.  \qedhere
\]
\end{proof}

The matrix $\Hess \ln (\Vert S \Vert)$ is non-positive because we are at a maximum point, and $\mathrm{Id} + M$ is positive definite, so the previous lemma yields:
\[
- 1 = - \langle \Hess \ln \zeta, \mathrm{Id} + M \rangle \ge 0
\]
which is absurd. Thus $E = \emptyset$, so $t = 1$ on $\SS^2$, and hence $S = 0$ by \eqref{eq555}. Therefore, in the case $t>0$ we finally get $U =\Id$. 

The solution $U=-\Id$ is obtained in the case $t<0$, as explained in the beginning of the proof.
\end{proof}

Recall that in the case $B\neq -\Id$ we defined $U=C^{-1}$ where $C=\frac{B+\Id}{2}$.
As a consequence of Proposition~\ref{s2}, we get the additional solutions $B = - 3 \mathrm{Id}$ or $B = \mathrm{Id}$. We obtain therefore three solutions to equations \eqref{sysend}, which lead to the endomorphism fields $A = \pm \mathrm{Id}$ and $A = \xi \otimes \xi - 3 P$, where we recall that $P$ is the orthogonal projector on $\xi^\perp$.

The previous analysis adapts as usual in the case where $\xi$ is right-invariant, yielding the fourth solution $A = - \xi \otimes \xi + 3 P$. 
\end{proof}

\subsection{Link with the sphere rigidity theorem}
A classical result due to Liebmann \cite{L99} states that the only isometric immersions of the round sphere $\SS^2$ in $\RR^3$ are the totally umbilical embeddings (hence they differ from the standard embedding by an isometry of $\RR^3$). Let us recall a property of Codazzi tensors in dimension $2$:
\begin{lemma}
Let $\Sigma$ be a surface endowed with a Riemannian metric $h$ and $S$ a field of endomorphisms on $\Sigma$. Then $S$ is a Codazzi tensor (i.e., $d^\nabla S=0$) if and only if $JSJ$ is divergence-free, where $J$ is the Hodge star on $1$-forms.
\end{lemma}
\begin{proof}
Let $X$ be a locally-defined unit vector field on $\Sigma$. Using $\nabla J=0$ we compute
\begin{align*}
-\delta^\nabla (JSJ) =& \nabla_X(JSJ)(X)+\nabla_{JX}(JSJ)(JX)\\
=&J\nabla_X(S)(JX)+J\nabla_{JX}(S)(JJX)\\
=&Jd^\nabla(S)(X,JX).\qedhere
\end{align*}
\end{proof}
Proposition \ref{s2} implies the following slight extension of the sphere rigidity theorem:
\begin{prop} 
Every isometric immersion of the round $2$-sphere in a flat $3$-manifold is totally umbilical. 
\end{prop}
Indeed, the Gauss and Codazzi equations of the embedding $\SS^2\hookrightarrow M$ with second fundamental form $S$ tell us that $\det(S)=1$ and $d^\nabla S=0$. By the above lemma, this is equivalent to $\det(JSJ)=1$ and $\delta^\nabla (JSJ)=0$, so by Proposition \ref{s2} we deduce that
$JSJ=\pm\Id$, which means that $S$ itself is $\pm\Id$.

We refer to \cite{AAR} for a modern proof of Liebmann's theorem using the curvature of the metric defined by the second fundamental form $S$.

The interplay between solving the system \eqref{sysend} and a nontrivial classical result might explain why the more general system \eqref{sd} is not so easy to solve.

\renewcommand{\refname}

\end{document}